\documentclass[12pt]{amsart}
\usepackage{amsfonts, amsthm, amssymb, amsmath, stmaryrd}
\usepackage{mathrsfs,array}
\usepackage{eucal,color,times,enumerate,accents}
\usepackage[all]{xy}
\usepackage{url}
\usepackage{enumitem}
\usepackage[bookmarksnumbered,bookmarksopen]{hyperref}
\hypersetup{
  colorlinks,
  citecolor=black,
  linkcolor=black,
  urlcolor=black}
  
  \usepackage[title]{appendix}
  
\usepackage{tikz-cd}
\usetikzlibrary{cd}
\usepackage{leftidx}
\usetikzlibrary{positioning}
\usepackage{dsfont}
\usepackage{tikz}
\usetikzlibrary{matrix,arrows,decorations.pathmorphing,positioning}

\usepackage[utf8]{inputenc}
\usepackage[T1]{fontenc}

\date{\today}

\newcommand\blfootnote[1]{%
\begingroup
\renewcommand\thefootnote{}\footnote{#1}%
\addtocounter{footnote}{-1}%
\endgroup
}
  
  \newcommand{\Addresses}{{% additional braces for segregating \footnotesize
  \bigskip
  \footnotesize

\textsc{CNRS, IMJ-PRG, Sorbonne Universit\'{e}, 4 place Jussieu, 75005 Paris, France}\par\nopagebreak
  \textit{E-mail address}, G.~Baldi: \texttt{baldi@imj-prg.fr},  \texttt{baldi@ihes.fr}  

  \medskip

  \textsc{University College London, Department of Mathematics, Gower street, WC1E 6BT, London, UK}\par\nopagebreak
   \textit{E-mail address}, R. Richard: \texttt{rodolphe.richard@normalesup.org}

   \medskip

\textsc{I.H.E.S., Université Paris-Saclay, CNRS, Laboratoire Alexandre Grothendieck. 35 Route de Chartres, 91440 Bures-sur-Yvette (France)}\par\nopagebreak
  \textit{E-mail address}, E. Ullmo: \texttt{ullmo@ihes.fr}

}}

\usepackage{bbm}

\usepackage{leftidx}

\usepackage[bottom=2.5cm, top=2.8cm, left=2.3cm, right=2.3cm]{geometry}

\input xy
\xyoption{all}

\addtocontents{toc}{\setcounter{tocdepth}{1}}

\usepackage{mathtools} 
\DeclarePairedDelimiter{\gen}{\langle}{\rangle} 
\newcommand{\eps}{{\varepsilon}}

\newcommand{\Gcal}{\mathcal{G}}
\newcommand{\etale}{\operatorname{\acute{e}t}}
\newcommand{\tens}{\mathbin{\otimes}}
\newcommand{\abs}[1]{\left\lvert{#1}\right\rvert}
\newcommand{\Witt}{\operatorname{W}}
\newcommand{\Wittk}{{}{\Witt(k)}{}}
\newcommand{\Wittp}{{}{\Witt(\FF_p)}{}}
\newcommand{\Frob}{\operatorname{Frob}}
\newcommand{\WittF}{\operatorname{F}_{\Wittk}}
\newcommand{\WittV}{\operatorname{V}_{\Wittk}}
\newcommand{\Dieukfull}{\Wittk[\WittF,\WittV]}
\newcommand{\Lbar}{\overline{L}}
\newcommand{\Acal}{\mathcal{A}}
\newcommand{\FrobAk}{\operatorname{Frob}_{A_k}}
\newcommand{\QQbar}{\overline{\mathbb{Q}}}
\newcommand{\CoHcrys}{{\operatorname H}^1_{\crys}}
\newcommand{\CoHet}{{\operatorname H}^1_{\etale}}
\newcommand{\Hcrys}{{\operatorname H}_1^{\crys}}
\newcommand{\Het}{{\operatorname H}_1^{\etale}}
\newcommand{\reduction}{\operatorname{red}}

\theoremstyle{plain}
\newtheorem{thm}{Theorem}[section]

\newtheorem{lemma}[thm]{Lemma}

\newtheorem{ex}[thm]{Example}

\newtheorem{prop}[thm]{Proposition}
\newtheorem{cor}[thm]{Corollary}
\theoremstyle{definition}

\newtheorem{rmk}[thm]{Remark}

\theoremstyle{remark}
\numberwithin{equation}{section}

\newcommand{\FF}{\mathbb{F}}

\newcommand{\Zar}{\operatorname{Zar}}

\newcommand{\Ators}{A_{\operatorname{tors}}}

\newcommand{\kbar}{\overline{k}}

\newcommand{\wt}[1]{\widetilde{#1}}

\DeclareMathOperator{\End}{End}
\DeclareMathOperator{\Hom}{Hom}
\DeclareMathOperator{\Gal}{Gal}

\DeclareMathOperator{\Spec}{Spec}

\DeclareMathOperator{\im}{Im}

\DeclareMathOperator{\Gl}{GL}

\newcommand{\crys}{\operatorname{crys}}
\newcommand{\Stab}{\operatorname{Stab}}

\newcommand{\N}{\mathbb{N}}

\newcommand{\Z}{\mathbb{Z}}
\newcommand{\Q}{\mathbb{Q}}
\newcommand{\Qp}{\mathbb{Q}_{p}}

\newcommand{\R}{\mathbb{R}}

\newcommand{\A}{\mathbb{A}}

\newcommand{\Oo}{\mathcal{O}}

\newcommand{\C}{\mathbb{C}}

\newcommand{\Qbar}{\overline{\mathbb{Q}}}
\newcommand{\Ff}{\mathbb{F}}

\newcommand{\ord}{\mathrm{ord}}

\begin{document}

\newcommand{\adjunction}[4]{\xymatrix@1{#1{\ } \ar@<-0.3ex>[r]_{ {\scriptstyle #2}} & {\ } #3 \ar@<-0.3ex>[l]_{ {\scriptstyle #4}}}}

\title{Manin--Mumford in arithmetic pencils}%\blfootnote{\emph{Date}. March 1, 2021.}
\author{Gregorio Baldi, Rodolphe Richard, and Emmanuel Ullmo}

\begin{abstract}
We present a refinement of Manin--Mumford (Raynaud's Theorem) for abelian schemes over some ring of integers. Torsion points are replaced by  \emph{special 0-cycles}, that is reductions modulo some prime (possibly varying) of Galois orbits of torsion points. There is a flat/horizontal part and a vertical one. The irreducible components of the flat part are given by the Zariski closure, over the integers, of torsion cosets of the generic fibre of the abelian scheme. The vertical components are given by translates of abelian subvarieties, which ``come from characteristic zero''. 
\end{abstract}

\maketitle

\tableofcontents

\section{Introduction}
\subsection{Motivation}
The following is often referred to as the \emph{Manin--Mumford conjecture}, stated in \cite{langdiv}. It was first proved by Raynaud \cite{raynaud1, raynaud2}, reproved then by Hindry \cite{hindry} and more recently by Pila--Zannier \cite{pilazannier}, among others.
\begin{thm}[Raynaud]\label{mm}
Let $A$ be an abelian variety over a field of $k$ characteristic $0$ and $E$ be a subset of torsion points of $A(\overline{k})$. Then the Zariski closure of $E$ is \emph{special}: a finite union of torsion cosets $a+B$ for some abelian subvariety $B\subseteq A_{\overline{k}}$  and a torsion point $a\in A(\overline{k})$.
\end{thm}

Theorem \ref{mm}, together with its multiplicative analogue (originally appeared in \cite{langdiv}), can be thought as the first instance of a more general philosophy, now referred to as \emph{Unlikely Intersections}. Recently Richard \cite{richard} has introduced a new net of conjectures aiming to understand unlikely intersections phenomena \emph{over the integers}. This includes new problems regarding the arithmetic of (models of) both abelian and Shimura varieties. The role of torsion points is replaced by the reduction to some residue field of positive characteristic of the Zariski closure, over the integers, of either a torsion or a CM-point. In an arithmetic family there are two perspectives:
\begin{itemize}
\item \emph{Horizontal} unlikely intersections. This case corresponds to an intersection with infinitely many vertical fibres;
\item \emph{Vertical} unlikely intersections. When the vertical fibres form a finite set. 
\end{itemize}
Even formulating the vertical part for Shimura varieties is challenging. Consider $Y_0(1)/\Z\cong\mathbb{A}^1_\Z$, the modular curve parametrising elliptic curves. Under the generalised Riemann hypothesis for quadratic inmaginary fields, Edixhoven and Richard \cite{basrod} proved the vertical André--Oort for $Y_0(1)\times Y_0(1)/\Z$ and later, unconditionally, Richard \cite{richard} proved the horizontal André--Oort for $Y_0(1)/\Z$. The above conjectures contain, as special cases, new statements for subschemes of abelian schemes over some ring of integers that can be thought as an arithmetic counterpart of Theorem \ref{mm}. The aim of this paper is to prove both a horizontal and vertical Manin--Mumford, as we explain in the next section. 

\subsection{Main results}\label{mainresultsec}
Let $K$ be a number field (with a fixed algebraic closure $\overline{K}$), $S$ a finite set of finite places of $K$ and $\Oo_{K,S}$ be the ring of $S$-integers of $K$. Write $\Gal (\overline{K}/K)$ for the absolute Galois group of $K$. Let $\mathcal{A}$ be an abelian scheme over (the spectrum of) $\Oo_{K,S}$. We denote by $A=A_K$ its generic fibre, and by $\mathcal{A}_\mathfrak{p}$ its special fibre at a closed point $\mathfrak{p}\in \Spec (\Oo_{K,S})$. The latter is an abelian variety over the residue field $\kappa(\mathfrak{p})$. In $\mathcal{A}$ we consider \emph{special 0-cycles}, that is the reduction to $\kappa(\mathfrak{p})$, of positive characteristic, of the Zariski closure in $\mathcal{A}$ of a torsion point $a\in A(\overline{K})$, for some prime ideal $\mathfrak{p} \in \Spec (\Oo_{K,S})$. Our main result describes the Zariski closure of a set of special 0-cycles in $\mathcal{A}$. To do so we introduce two classes of subschemes of $\mathcal{A}$:
\begin{itemize}
\item \emph{Flat cycles}. Those obtained as the Zariski closure in $\mathcal{A}$ of special subvarieties of $A_{\overline{K}}$ (in the sense of Theorem \ref{mm}), and
\item \emph{Torsion cycles}. The translates of an abelian subvariety of some vertical fibre $\mathcal{A}_{\mathfrak{p}}$ by a $\overline{\kappa(\mathfrak{p})}$-point of $\mathcal{A}$.
\end{itemize}
Intersections between flat cycles and vertical fibres are particular cases of finite unions of torsion cycles, but in general torsion cycles are given by an abelian subvariety of $\mathcal{A}_\mathfrak{p}$ which is not, a priori, coming from a flat subvariety of $\mathcal{A}$. A closed subscheme of $\mathcal{A}$ is called \emph{special} if it can be written as a finite union of flat and torsion cycles. %Moreover flat cycles can be globally (additively) torsion.
\begin{thm}\label{mainthm}
Let $\mathcal{V}\subset \mathcal{A}$ be a closed subscheme. If $\mathcal{V}$ contains a Zariski dense sequence $(E_n)_{n\geq 0}$ of special 0-cycles then $\mathcal{V}$ is special. 
\end{thm}
It is natural to try to understand when the torsion cycles given by Theorem \ref{mainthm} are intersections between flat cycles and vertical fibres, but there is a \emph{$p$-adic obstruction} to this, as the next example shows.
\begin{ex}\label{anexample}
Let $A$ be a geometrically simple abelian $K$-surface, and $\mathfrak{p}$ a prime of residue characteristic $p>0$, such that $\mathcal{A}_\mathfrak{p}$ is isogenous to $\mathfrak{B} \times \mathfrak{B}'/\kappa(\mathfrak{p})$, with $\mathfrak{B}$ ordinary and $ \mathfrak{B}'$ super-singular, that is $\mathfrak{B}'(\overline{\kappa(\mathfrak{p}}))[p^\infty]\cong 0$. Let $(a_n)_n \subset A(\overline{K})$ be a sequence of points such that $\ord (a_n)=p^n$. Since $ \mathfrak{B}'[p^\infty]=0$, the reduction to $\overline{\kappa(\mathfrak{p})}$ of the $\Gal(\overline{K}/K)$-orbit of the $a_n$'s is contained in $\mathfrak{B}\times\{0_{\mathfrak{B}'}\}$. We can choose the $a_n$'s in such a way that the Zariski closure of $\bigcup_n E_n$ is $ \mathfrak{B}\times \{0_{\mathfrak{B}'}\} \subset \mathcal{A}_\mathfrak{p}$ (as predicted by Theorem \ref{mainthm}), but $ \mathfrak{B}\times \{0_{\mathfrak{B}'}\}$ can not come from characteristic zero, since $A_{\overline{K}}$ is simple.
\end{ex}

Our second main result considers only special 0-cycles of \emph{order coprime to the residue characteristics}, that is cycles $E$ associated to a torsion point $a$ and residue field $\kappa(\mathfrak{p})$ such that $\ord (a)$ is coprime with $\operatorname{char}(\kappa(\mathfrak{p}))$.
\begin{thm}\label{mainthmlift}
If $(E_n)_{n\geq 0}$ is a sequence of special 0-cycles in $\mathcal{A}$ of order coprime with the residue characteristic, then the Zariski closure of $\bigcup_{n\geq 0} E_n $ is a finite union of flat cycles and intersections between flat cycles and vertical fibres.
\end{thm}
Each special $0$-cycle $E_n \subset \mathcal{A}$, set theoretically, can be written as the image of $\Gal(\overline{K}/K) \cdot a_n$, for a family of points $a_n\in A_{\text{tors}}(\overline{K})$ along a\footnote{Even if there is no canonical map $\pi _{\mathfrak{p}_n}: A(\overline{K})\to \mathcal{A}(	\overline{\kappa(\mathfrak{p}_n)})$, the image of a Galois orbit in $ \mathcal{A}(	\overline{\kappa(\mathfrak{p}_n)})$ is well defined (i.e. does not depend on a particular choice of $\pi_{\mathfrak{p}_n}$). Indeed it is just the intersection between the vertical fibre $\mathcal{A}_\mathfrak{p}$ and the Zariski closure over $\Oo_{K,S}$ of $a_n$. We remark here that, when restricted to the prime-to-$\operatorname{char} \kappa(\mathfrak{p})$ part of the torsion, $\pi _{\mathfrak{p}}$ is the isomorphism identifying $\pi_1^{et}(\Oo_{K,S})$-equivariantly two geometric fibres of a locally constant constructible sheaf.} map
\begin{displaymath}
\pi_{\mathfrak{p}_n}: A(\overline{K})\to \mathcal{A} (	\overline{\kappa(\mathfrak{p}_n)})
\end{displaymath}
to a family of fields of positive characteristic $\kappa(\mathfrak{p}_n)$. Notice that the result is purely \emph{vertical} when the set $\{\mathfrak{p}_n\}_{n \geq 0}$ is finite. It may be helpful to give a reformulation of Theorem \ref{mainthmlift}, assuming that all special $0$-cycles lie in a fixed vertical fibre $\mathcal{A}_\mathfrak{p}$:
\begin{thm}\label{lastthm}
Fix a prime $\mathfrak{p}\subset \Oo_{K,S}$ of residue characteristic $p$. Let $\mathcal{A}$ be an abelian scheme over $\Oo_{K,S}$, $(a_n)_n\subset A_{\operatorname{tors}}$ be a sequence of torsion points of order coprime with $p$, and write $E_n=\pi_{\mathfrak{p}}(\Gal(\overline{K}/K) \cdot a_n)$, for a map $
\pi_{\mathfrak{p}}: A(\overline{K})\to \mathcal{A}_\mathfrak{p}(\overline{\kappa(\mathfrak{p})}).$ 
Then every component of $\mathcal{V}$, the Zariski closure of 
\begin{displaymath}
\bigcup_{n\geq 0} E_n\subset \mathcal{A}_\mathfrak{p}\subset \mathcal{A},
\end{displaymath}
is a component of the intersection between a flat cycle $\mathcal{B}\subset \mathcal{A}$ and $\mathcal{A}_\mathfrak{p}$.
\end{thm} 
An interesting and simple corollary of Theorem \ref{lastthm} is the following.
\begin{cor}\label{cor 1.2.5}
With the notations and the assumptions of Theorem \ref{lastthm}, if the sequence of subsets $(\Gal (\overline{K}/K)\cdot a_n)_n$ is Zariski dense in $A_K$, then the special 0-cycles $(E_n)_{n\geq 0}$ are Zariski dense in $\mathcal{A}_\mathfrak{p}$.
\end{cor}

The final main result of the paper shows that the ``$p$-rank zero factor'' explained in Example \ref{anexample} is indeed the only obstruction to the lifting problem.
\begin{thm}\label{pureppart}
Let $ \Oo_{K,S}, \mathfrak{p}, p, \pi_\mathfrak{p} $ and $\mathcal{A}$ be as in Theorem \ref{lastthm}. If $(a_n)_n\subset A_{\operatorname{tors}}$ is a sequence of torsion points (possibly of order divisible by $p$), then every component  of the Zariski closure $\mathcal{V}$ of
\begin{displaymath}
\bigcup_{n\geq 0}\pi_{\mathfrak{p}}(\Gal(\overline{K}/K) \cdot a_n) \subset \mathcal{A}_\mathfrak{p}\subset \mathcal{A},
\end{displaymath}
is of the form $\alpha + \mathfrak{B}$, and there exists an abelian subvariety $B \subseteq A$ 
(over $\overline{K}$) 
such that $ \mathfrak{B} \subset B_\mathfrak{p}$ and the quotient $B_\mathfrak{p}/  \mathfrak{B}$ is of $p$-rank zero.
\end{thm} 

\begin{rmk}\label{rmkimplication}
Any of Theorem \ref{mainthm} or \ref{mainthmlift} implies the classical Manin--Mumford Theorem \ref{mm} (at least after it is reduced to $k=\Qbar$). Let indeed be $V\subset A_{\Qbar}$ be a component of the Zariski closure of a sequence of torsions points $(a_n)_n$. Consider the special zero cycles $E_{m,n}$ given by reducing modulo $\mathfrak{p}_m$ (the $m$th prime) the Galois orbit of $a_n$. By the horizontal Manin--Mumford (i.e. Theorem \ref{mainthm} for the flat components), we have a torsion coset $B\subset A_{\Qbar}$ such that
\begin{displaymath}
E_{m,n}\subseteq B \subseteq V \subseteq A_{\Qbar}.
\end{displaymath}  
Repeating this on $A/B$ concludes the argument. The same proof works with the strong vertical Manin--Mumford established in Theorem \ref{mainthmlift}, choosing carefully the reduction place (in both cases we need a characteristic zero abelian subvariety of $A$).
\end{rmk}
Hindry's approach to the Manin--Mumford conjecture turned out to provide a fruitful strategy for the \emph{André--Oort conjecture}, as first envisioned by Edixhoven \cite{bas}. We refer to \cite{MR3821177, zbMATH06792103} for an overview of recent results and progress. We hope that our strategy can be applied to obtain new evidences in favour of Richard's arithmetic conjectures. Finally we remark that there are also variants of the classical Manin-Mumford in positive characteristic by Hrushovsky, Pink-Rossler and Scanlon (see for example \cite{zbMATH05004165} and references therein). They appear to be interested in a different problem, since here it is crucial to start with a an abelian scheme over some ring of integers and take Galois orbits under the absolute Galois group of a number field.
\subsection{Strategy of the proof and outline of the paper}
The proof of the horizontal part of Theorem \ref{mainthm} is inspired by the approach of Hindry \cite{hindry} to his proof of Theorem \ref{mm}. Our arithmetic variant, however, presents new phenomena and new difficulties, as Remark \ref{rmkintro} explains. The strategy for Theorem \ref{mainthmlift} and Theorem \ref{pureppart} is more complicated, and relies on the Tate conjecture (both the $p$-adic and mod $p$ variants).

\subsubsection{Horizontal and weak vertical}\label{section111}
Let $c$ be a positive integer and $L=L(c)$ be Lang group of $c$-th powers \emph{homotheties}, as in \eqref{langdef}. It acts on each $\mathcal{A} (\overline{ \kappa (\mathfrak{p})})$. As for Hindry's proof, the only ``Galois input'' we need to know is the comparison explained in Remark \ref{remarkaction} between the $\Gal(\overline{K}/K)$ and $L$ actions on torsion points. In particular the following implies Theorem \ref{mainthm}.
\begin{thm}\label{thm22} Let $\mathcal{V}\subseteq \mathcal{A}$ be a closed $\Oo_{K,S}$-subscheme of $\mathcal{A}$ which admits a Zariski dense set of closed points $(a_n)$ such that $L\cdot a_n\subseteq \mathcal{V}$. Then $\mathcal{V}$ is special.
\end{thm}
The strategy is based on Proposition \ref{keyprop}. From a subvariety $V$ of $\mathcal{A}_{\mathfrak{p}_n}$ of degree $D$ and dimension $\Delta$, containing $L\cdot a_n\subseteq \mathcal{A}_{\mathfrak{p}_n} $, we show that, if the $\ord (a_n)$ is bigger than some function $f=f(D, \Delta)$, then there exists $0\neq B\subseteq \mathcal{A}_{\mathfrak{p}_n}$ an abelian subvariety of $\mathcal{A}_{\mathfrak{p}_n}$ such that $a_n+B \subset V$ and $\delta := \ord (a+B)$ (its order in $\mathcal{A}_{\mathfrak{p}_n} / B$ ) is smaller than $\ord (a_n)$. After repeating the argument in $\mathcal{A}_{\mathfrak{p}_n} / B$, we may assume that $\dim B=0$ and that $\ord (a_n)$ is bounded by $f$. 

\begin{rmk}\label{rmkintro}
If a torsion coset $a+ A'$ of $A$ is contained in a subvariety $V\subseteq A$ and in $[q](V)$, for some $q>1$, then 
\begin{displaymath}
a+A' \subseteq V \cap [q](V)
\end{displaymath}
and in particular there is an irreducible component $V'$ of $V$ such that $a+A' \subseteq V'$. Compared with Hindry's approach, where everything is defined over a number field $K$, the main difference is that there
\begin{displaymath}
\Gal(\overline{K}/K)\cdot a + A' \subset V'.
\end{displaymath}
In an arithmetic pencil, this point becomes much more subtle, since the standard habit of reducing to irreducible components is usually incompatible with the problem. Indeed the main new difficulty we have to face is that we do not know, a priori, that every component of the reduction of $V$ mod $\mathfrak{p}_n$ contains many elements of the special 0-cycle $E_n$, even if the cardinality of $E_n$ is tending to infinity.
\end{rmk}

\subsubsection{Vertical and lifting}\label{strategyvertical}
Theorem \ref{mainthmlift} requires different ideas and more sophisticated tools. It turns out that, using various forms of the Tate conjecture (Faltings theorem, see Section \ref{sectiontate}), the only abelian subvarieties we have to consider are the ones coming from characteristic zero. This is interesting because there are many more abelian subvarieties of $\mathcal{A}_\mathfrak{p}$, than the one coming from characteristic zero. The analogy with Shimura varieties is less clear: what is a \emph{special subvariety} of the reduction mod $p$ of a Shimura variety, which does not come from characteristic zero (that is from reducing mod $p$ a Shimura subvariety)?

We briefly explain our strategy for Theorems \ref{mainthmlift} and \ref{lastthm} (with the same notation as in Theorem \ref{lastthm}). By looking at the order of the $a_n$, there are essentially two possibilities to keep in mind: 
\begin{enumerate}
\item $\ord(a_n)=\ell^n$, for all $n$ (where $\ell$ is a prime different from $p$);
\item $\ord (a_n)$ is the $n$-th prime number, for all $n$.
\end{enumerate}
The main protagonists are (considered as $\Gal(\overline{K}/K)$-representation):
\begin{itemize}
\item The $\ell$-adic (rational) Tate module of $A$ (for $\ell\neq p$);
\item The adelic Tate module of $A$ (or quotients thereof).
\end{itemize}
Depending only on the order of the $a_n$, we use the $\ell$-adic Tate module in (1), and the adelic one in (2). The idea in both cases is similar, even if it may require different tools, and can be summarised as follows. Let $V$ be one of the above \emph{semisimple} $\Gal(\overline{K}/K)$-representations. The Galois orbits of the points $a_n$ determine a subrepresentation $W\subset V$, such that $V=W \oplus W'$. Faltings's theorem asserts that such a decomposition is induced by an element $u \in \End(A)\otimes E_V $, where $E_V \in \{\Q_\ell, \text{ a quotient of }\widehat{\Z}\}$, depending on $V$. Theorem \ref{mainthm} asserts that such a decomposition is, ``after reducing mod $\mathfrak{p}$'', induced by an element of $\End (A _{\mathfrak{p}})$, and we prove that $u$ can be found in $\End(A)\otimes \Q$ (thanks to the formal descent lemma \ref{keylemma}). This corresponds to lifting a torsion cycle to a flat one. 

\subsubsection{Points of $p$-torsion in characteristic $p$}\label{sectionstrategyp} The proof of Theorem \ref{pureppart} is inspired by the $\ell$-adic argument presented above, but it is necessarly more complicated, since it has to take into account Example \ref{anexample}. To emulate the argument of Section \ref{strategyvertical}, we need to find a vector space where both $\End(A)\otimes \Q_p$ and $\End(\mathcal{A}_\mathfrak{p})\otimes \Q_p$ act faithfully and compatibly. This is done using some comparison theorem from $p$-adic Hodge theory, but it requires to extend the coefficients from $\Q_p$ to one of Fontaine's period rings (see Proposition \ref{thingy statement}). Such results are presented in a self contained way in the final Appendix \ref{preliminaries}.

\subsection*{Acknowledgements}
The second author would like to thank the I.H.E.S. for a research visit in September and October 2020, during which most of the work regarding the horizontal Manin--Mumford was done. Rodolphe Richard was supported by the Leverhulme grant RPG-2019-180. Finally we thank the referees for their careful reading and comments.

\section{Preliminaries and Hindry's method}\label{prelimnaries}
In this section we first reduce the main results to the case where $\End(A_K)=\End(A_{\overline{K}})$. We then define the Lang group and recall its relation to the Galois action on torsion points. Finally we review some results which appear in Hindry's method for Manin--Mumford \cite{hindry}, and in the process we check that Hindry's results hold true also over fields of arbitrary characteristics.

\subsection{A first reduction}\label{reduction}
We explain here why the statements of the results of Section \ref{mainresultsec} are not affected by finite field extensions $K'/K$. This may look surprising at first since there can be many more $K'$-abelian subvarieties of a $K$-abelian variety $A$, rather than $K$-abelian subvarieties.

Without loss of generality, we may assume that $K'/K$ is Galois. Let $(E_n)_{n\geq 0} \subset \mathcal{A}$ be a sequence of special zero cycles obtained as the $\Oo_{K,S}$-Zariski closure of a sequence of torsion points $a_n\in A(\overline{K})$, and $(E'_n)$ be the $\Oo_{K',S'}$-Zariski closure of the $a_n$ (for some set of $K'$ places $S'$ above $S$). Let $\mathcal{V}\subset \mathcal{A}$ (resp. $\mathcal{V}'\subset \mathcal{A}_{\Oo_{K',S'}}$) be the Zariski closure of the $E_n$ (resp. of the $E'_n$). The closure of $\mathcal{V}'$ over $\Oo_{K,S}$ is $\mathcal{V}$. In particular if $\mathcal{V}'$ is special or even a finite union of flat cycles and intersections of flat cycles and vertical fibres, so is $\mathcal{V}$. Finally fix $\mathfrak{p}\subset \Oo_{K,S}$ and $\mathfrak{p}' \subset \Oo_{K',S'}$ above $\mathfrak{p}$. If every $K'$-component of $\mathcal{V'}$ $F+\mathcal{B}$ (for some finite set $F\subset \mathcal{A}$ and abelian $\Oo_{K',S'}$ subscheme of $\mathcal{A}$) there exists an abelian subvariety $B' \subset A_{K'}$ such that $\mathcal{B}\subset B_{\mathfrak{p}'}$ of $p$-corank zero, then the same holds for the $K$-components of $\mathcal{V}$.

From now on, we may and do assume that $\End(A_K)=\End(A_{\overline{K}})$.

\subsection{Lang group and adelic Tate module}\label{langsection}
Let $K$ be a number field, and $A$ be a $g$-dimensional abelian variety over $K$ (for some $g>0$). Denote by $T(A)$ the $\Gal(\overline{K}/K)$-module obtained by taking the inverse limit of the $A[n]$'s, that is the (integral) \emph{adelic Tate module} of $A$. We write $\widehat{\Z}$ for the profinite completion of the ring of integers $\Z$, and notice that $T(A)\cong \widehat{\Z}^{2g}$. The proof of the following can be found for example in \cite[Thm. B]{cadoretmoonen}. 
See also \cite[Cor. 1, page 702]{bog}, \cite[Sec. 2 (a)]{hindry}, \cite{zbMATH01837199}, and references therein. For any $c\in \Z_{\geq 1}$, we define the \emph{Lang group} $L=L(c)$ by the formula
\begin{equation}\label{langdef}
L=L(c):=\ker \left( \widehat{\Z}^*\to \widehat{\Z}^* \otimes \Z /c\Z \right),
\end{equation}
seen as a subgroup of $\Gl(T(A))$. By definition, $L=L(c)$ acts on the torsion points of $A$. 
\begin{thm}[Bogomolov--Serre]\label{lang}
There exists a positive integer $c=c(A,K)$ such that the image of the map describing the action of $\Gal(\overline{K}/K)$ on the torsion points,
\begin{displaymath}
\rho_A : \Gal(\overline{K}/K)\to \Gl(T(A)),
\end{displaymath}
contains the $c$-th powers of the \emph{homotheties}. That is $L(c) \subset \im (\rho_A). $
\end{thm}

\begin{rmk}\label{remarkaction}
As a corollary of Theorem \ref{lang}, we have the following. Let $a\in A_{\operatorname{tors}} (\overline{K})$ be of order $d$, and $l$ a positive integer prime to $d$. Then there exists $\sigma \in \Gal(\overline{K}/K)$ such that
\begin{displaymath}
\sigma\cdot a=\sigma(a)=[l^c](a).
\end{displaymath}
\end{rmk}

\subsection{Degree, stabiliser subgroup, and a criterion for specialness}\label{degreformulas}
We discuss Hindry's results of Bézout type in abelian varieties (see \cite[Sec. 1]{hindry}). The first use of a Bézout theorem/intersection theory to approach Lang conjecture can be dated back to Tate's proof \cite[Sec. 2 (page 231)]{langdiv} of a case of the multiplicative Manin--Mumford. See also the discussion about abelian varieties in Section 3 of \emph{op. cit.}.

Let $k$ be an algebraically closed field of characteristic $p\geq 0$, and $A$ be an abelian variety (of dimension $>1$) over $k$.
\begin{lemma}\label{lemmahin}
Let $V$ be an irreducible subvariety of $A$, and let $B\subset A$ be its stabiliser. Let $q$ be a positive integer not divisible by $p$. We have
\begin{equation}\label{formuallemma}
\deg([q](V))=q^{2\dim(V)}\deg(V)\cdot (\# B[q])^{-1}.
\end{equation}
\end{lemma}

\begin{proof}
The proof presented in \cite[Lem. 6]{hindry} applies also to cover the above statement, if $p\nmid q$.
\end{proof}

For simplicity, form now, we simply write that $q$ is prime to $p$, to mean that $q$ is non zero and not divisible by $p$ (which makes sense both when $p=0$ and $p>0$).
\begin{cor}\label{corhin}Let $ A, V$, and $B$ as in Lemma \ref{lemmahin}. Let $q, q' $ be positive integers, prime to each other and to $p$, such that
\begin{displaymath}
[q](V)=[q'](V).
\end{displaymath}
Then 
\begin{displaymath}
\# B[q]=
q^{2\dim(V)}, \text{  }\# B[q']=
q'^{2\dim(V)},\text{ and }\# B[qq']=
(qq')^{2\dim(V)}.
\end{displaymath}
\end{cor}
\begin{proof}
From Lemma \ref{lemmahin} we have that
\begin{displaymath}
q^{2\dim V}\# B[q]^{-1}=q'^{2 \dim V} \# B[q']^{-1}.
\end{displaymath}
Since $q$ and $q'$ are coprime, the only possibility is that $\# B[q]=q^{2\dim(V)}$ and $\# B[q']=q'^{2\dim(V)}$. Again by the fact that $q$ and $q'$ are coprime, we have $B[qq']=B[q]\times B[q'],$
concluding the proof.
\end{proof}

\begin{lemma}\label{lemmadegree} Let $ A, V$, and $B$ as in Lemma \ref{lemmahin}. Denote by $B^0$ the neutral component of $B$, and the group of components of $B$ by $\pi_0(B)=B/B^0$. We have
\begin{displaymath}
\#\pi_0(B)\leq \deg(B)=\#\pi_0(B)\cdot \deg(B^0)\leq \deg(V),
\end{displaymath}
and 
\begin{displaymath}
\#B[q]=\#\pi_0(B)[q]\cdot\#B^0[q]\leq\#\pi_0(B)\cdot q^{2\dim(B)}.
\end{displaymath}
Assume now that $\#B[q]=q^{2\dim(V)}$. Then
\begin{displaymath}
q^{2(\dim(V)-\dim(B))}\leq \#\pi_0(B)\leq \deg(V).
\end{displaymath}
If moreover
\begin{displaymath}
\deg(V) < q^2,
\end{displaymath}
then $V$ is special.
\end{lemma}
\begin{proof}
The first formula follows from \cite[Lem. 8]{hindry}. For the second formula, just notice that $\#B^0[q](\overline{k})=q^{2\dim B^0}$, as long as $q$ is prime to $p$. The third formula is a direct consequence of the previous two.

Regarding the last conclusion, we necessarily have $\dim(V)-\dim(B)=0$, since otherwise 
\begin{displaymath}
q^{2(\dim(V)-\dim(B))} \leq \#\pi_0(B)\leq \deg(V) < q^2.
\end{displaymath}
Choose a point $v\in V$, and notice that necessarily $v+B\subset V$. As $\dim(v+B)=\dim(B)=\dim(V)$, and $V$ is irreducible, we have that $V=v+B$ is special.
\end{proof}

\begin{prop}\label{Hindry Lemma Criterion} Let $V$ be an irreducible subvariety of $A$, $q$ and $q'$ be prime to each other and to $p$ be such that
\begin{displaymath}
[q](V)=[q'](V), \text{  and  } \deg(V)^{1/2}<q\cdot q'.
\end{displaymath}
Then $V$ is a special subvariety of $A$, and so is $[q](V)$.
\end{prop}

Notice that the condition $\deg(V)^{1/2}<q\cdot q'$ is actually implied by $\deg(V)^{1/4}\leq q<q'$. We only use the latter condition. 
\begin{proof}Corollary \ref{corhin} implies that $\#B[qq']=(qq')^{2\dim(V)}.$
Therefore the assumption on $\#B[qq']$ (with $qq'$ for $q$) in Lemma \ref{lemmadegree} is satisfied and so $V$ is special. 
\end{proof}

Finally we check for reference that Hindry's criterion \cite[Lem. 15]{hindry} stays valid in any characteristics, and that it applies to subvarieties which are not necessary irreducible. See also \cite[Cor. on page 703]{bog}.
\begin{cor}[Hindry] \label{corhindry}Let $V$ be a subvariety of $A$. If $V=[q](V)$, for some $q>1$ prime to $p$, then each component of $V$ is special.
\end{cor}

\begin{proof} Let $r\leq \deg(V)$ the number of components of $V$, and $r!$ be the cardinality of the symmetric group acting on a set with $r$ elements. Since $[q]$ permutes the components of $V$, every component $Z$ is stable under $[q^{r!}]$, and similarly under $[m]$, for $m:=q^{\deg(V)!}$. Of course $m^2> \deg(V)$, and Proposition \ref{Hindry Lemma Criterion} on $Z$, $m$ and $q'=1$ applies to conclude. 
\end{proof}

\section{Horizontal and weak vertical, Theorem \ref{mainthm}}\label{mainsection} From now on we introduce some novelties in the Hindry's method, since, as explained in Remark \ref{rmkintro}, the standard way of applying it is incompatible with the various reductions mod $p$. For example we have to use the prime number theorem, among other things, to choose the primes carefully.

\subsection{Statement of the main proposition}\label{statementprop}Let $L=L(c)$ be the Lang group associated to some positive integer $c$, as defined in Section \ref{langsection}. The next result is the key for proving Theorem \ref{thm22}, and it is proven in Section \ref{proofmainprop}. Here we consider a field $k$ (with a fixed algebraic closure $\overline{k}$) of characteristic $p\geq 0$, and a $g$-dimensional abelian variety $A$ over $k$ (for some $g>1$).

\begin{prop}\label{keyprop}
Let $V$ be a subvariety of $A$ (not necessarily irreducible) of degree $D$ and dimension $\Delta$, and let $a$ be a torsion point of $A(\overline{k})$ of order $d$. Assume that
\begin{equation}\label{Hypotheses}
L\cdot a\subseteq V.
\end{equation}
Then there exist $\alpha \in A_{\operatorname{tors}}$ and $B\subset A$ an abelian subvariety of $A$ such that
\begin{displaymath}
L \cdot a \subseteq L\cdot \alpha+B\subseteq V, \text{  and   } \ord(\alpha+B)\leq \delta. 
\end{displaymath}
Where $\delta= \delta(D,\Delta,c)$ is a constant depending only on $D$, $\Delta$, and $c$. 
\end{prop}
The strategy is to replace $V$ by a special subvariety $Z\subset V$, still containing $L\cdot a$, and of degree as small as possible. In this case one can hope that not only $L\cdot a$, but most of the subgroup of $A$ generated by $a$, which we denote by $<a>,$ is contained in $Z$ (and hence in $V$). 
\begin{rmk}
The proof implies a stronger statement concerning torsion cosets $a+A'$ of $A$, rather than just torsion points.
\end{rmk}

\subsection{Proof of Proposition \ref{keyprop}}\label{proofmainprop}
The proof consists of two lemmas, which are proven below. We first introduce some notation. Given a positive integer $d$, we write:
\begin{itemize}
\item $p(d)$ for the $d$-th prime number;
\item $\omega(d)$ for the number of distinct prime divisors of $d$;
\item $g(d)$ for the \emph{Jacobsthal's function}: $g(d)$ is the smallest number $M$ such that every sequence of $M$ consecutive integers contains an integer coprime to $d$.
\end{itemize}
The Jacobsthal's function plays a role in providing sharp estimates. The reader not interested in explicit estimates may replace these arguments by less precise, but easier, estimates that will be explained during the proofs, indeed we only need to know that $g(d)=d^{o(1)}$.

\begin{prop}\label{jacob}Let $d, d',a$ be positive integers and $d' | d$. If $a$ is prime to $d'$ and
\begin{equation} \label{invertible lift}
(a+kd' )\wedge d=1,
\end{equation}
for some $k\in\Z$, then we can find $k'\geq 0$ satisfying \eqref{invertible lift}, and such that
\begin{displaymath}
0\leq k'< g(d)=d^{o(1)}.
\end{displaymath}
In the case $0\leq a<d'$, we have $a+k'\cdot d' \leq  g(d)\cdot d'$.
\end{prop}
Iwaniec \cite{iw} proved that 
\begin{displaymath}
g(d)\leq C \cdot \omega(d)^2 \cdot  \log \omega(d)^2,
\end{displaymath}
for some unknown $C$. The best known explicit upper bounds
\begin{equation}\label{explicit}
g(d)\leq 2^{\omega(d)},  \text{  and} \qquad g(d)\leq 2 \omega(d)^{2+2e \log \omega(d)},
\end{equation}
are due to Kanold \cite{kanold} and Stevens \cite{stevens}. In particular, $g(d)=d^{o(1)}$

\begin{proof}[Proof of Proposition \ref{jacob}] Set $n=\frac{d}{d'}$. We are looking for the smallest $k' \in \N$ such that 
\begin{displaymath}
(a/d'+k') \wedge n=1
\end{displaymath}
where by $a/d'$ denotes $a\cdot d'^\ast$ for $d'^\ast+n\Z$ an inverse of $d'+n\Z$ in $\Z/(n)^\times$. By definition of $g(d)$, this happens for some $0\leq k'< g(d)$. 
\end{proof}

\subsubsection{First lemma}
The first step towards Proposition \ref{keyprop} consists in finding a subvariety $V'\subset V$ of degree $d^{o(1)}$. %From now on we denote by $o(1)$ some functions depending only on $c$, and $\Delta$. To be more precise/explicit our $o(1)$ depends only on $A$, via $g=\dim (A)$ and $h(A)$, the semistable (logarithmic) Faltings height of $A/K$ (see Remark \ref{effectivec}).
\begin{lemma}\label{lemma1}
 Let $V$ be a subvariety of $A$ of degree $D$ and dimension $\Delta$, and $A'$ and abelian subvariety of $A$ such that 
 \begin{equation}\label{Hypotheses2}
L\cdot a+A'\subseteq V.
\end{equation}
We have exactly one of the following.
\begin{itemize}
\item A component of maximal dimension $X$ of $V$ is a special subvariety, say of the form $\alpha+B$, for some torsion point $\alpha\in A$ and $B$ an abelian subvariety of $A$, such that for every $l\in L$, $l\alpha+B$ is a component of $V$ and the order of each $l\alpha+B$ is bounded by $D^{O(1)}\cdot d^{o(1)}$, as $d\to + \infty$, and 
\begin{displaymath}
L \cdot a+A' \subseteq L\cdot \alpha+B\subseteq V.
\end{displaymath}
\item There exists $V'\subset V$ such that $\dim(V')<\dim(V)$, $V'$ satisfies \eqref{Hypotheses2}, and
\begin{equation}\label{meet condition}
\text{every component of $V$ contains a component of $L\cdot a+A'$},
\end{equation}
and has degree bounded as follows:
\begin{equation}\label{bound V'}
\deg(V')\leq  D^{O(1)}\cdot d^{o(1)},
\end{equation}
as $d\to + \infty$.
\end{itemize}
\end{lemma}
More precisely and explicitly we prove that
\begin{equation}\label{def f D d}
\deg(V')\leq f(D,d):= D^2\cdot \left(p(\left\lceil{\sqrt[2c]{D}}\right\rceil+D+\omega(d)+1)^c\cdot g(d \cdot \max\{1,p\})\right)^{2\cdot c\cdot\Delta},
\end{equation}
where $\left \lceil{\sqrt[2c]{D}}\right\rceil$ denotes the smallest integer bigger or equal than $\sqrt[2c]{D}$.

Two observations that follow directly from the assumptions \eqref{Hypotheses2}-\eqref{meet condition}, and are also at the heart of Hindry's method:
\begin{itemize}
\item For every integer $m$ prime to $d=\ord (a+A')$, we have $V\supseteq  L\cdot a+A'=m^c\cdot L\cdot a+A'\subseteq [m^c](V)$. Hence
\begin{displaymath}
L\cdot a+A'\subseteq V\cap  [m^c](V);
\end{displaymath}
\item For every component $X$ of $V$, we have
\begin{displaymath}
\emptyset \neq X\cap (L\cdot a+A')\subseteq X\cap [m^c](V).
\end{displaymath}
In Lemma \ref{lemma1} we did not assume that \eqref{meet condition} holds, but the statements is still true, since we may discard a finite number of components of $V$, to assume that \eqref{meet condition} holds.
\end{itemize}

\begin{proof} 
Let
\begin{equation}\label{defx}
 x:=\left\lceil{\sqrt[2c]{D}}\right\rceil+D+\omega(d)+1, \text{  and  }n:=p(x).
\end{equation}
For simplicity one can also simply replace $x$ by $x_D:=2D+\omega(d)+1$, just by noticing that $x\leq x_D$. This is indeed what we will do in the second Lemma. Since $x$ is at least $4$, we can use some non-asymptotic bounds on the prime-counting function and write
\begin{equation}\label{eqonprimes}
n=p(x)\leq x \log (x)(1+\log\log(x)),
\end{equation}
see for example \cite[Thm. 3, page 69]{rosser1962}. Finally we set
\begin{displaymath}
N:=n^c\cdot g(d \cdot \max \{1,p\}),
\end{displaymath}
where $g(-)$ denotes the Jacobsthal's function, as introduced in Section \ref{proofmainprop}. The reason why the factor $\max \{1,p\}$ appears is to deal simultaneously the cases of zero and positive characteristic.

 Let $X$ be a component of $V$ of maximal dimension, and set
\begin{displaymath}
X'=X'(m):=X\cap [m^c](V),
\end{displaymath}
for some $m=m(X)$ chosen in the following set
\begin{equation}\label{setS}
\Sigma=\Sigma(d,D):=\{m^c \text{ such that }m\in\{1,\ldots , N	\}, m\wedge d=1, \text{ and }p\nmid m \}.
\end{equation}
Assume now that there exists $m^c\in \Sigma$ such that $\dim(X')<\dim(X)$. Then the degree formula from Section \ref{degreformulas}, namely \eqref{formuallemma}, and the fact that $m^c\leq N^c=n^{c^{2}} \cdot g(d\cdot \max \{1,p\})^c$, imply that
\begin{displaymath}
\deg(X')\leq \deg(X)\cdot \deg([m^c](V))\leq \deg(X)\cdot \deg(V)\cdot N^{2\cdot c\cdot \Delta}.
\end{displaymath}
Assume now that for each component $X$, of maximal dimension, there exists $m^c\in \Sigma$ such that $\dim(X'=X'(m))<\dim(X)$, and let $V'$ be the union the $X'$s and the lower dimensional components of $V$. In a similar way we can control the degree of $V'$:
\begin{displaymath}
\deg(V')\leq \deg(V)^2\cdot N^{2\cdot c\cdot \Delta}.
\end{displaymath}
The above, combined with the explicit bounds for the Jacobsthal's function discussed in Section \ref{proofmainprop}, give the second conclusion.

We are left to the case where the above assumption fails. That is the desired $m\in \Sigma$ can not be found for some $\Delta$-dimensional  component $X$ of $V$. In other words
\begin{equation}\label{eqqq}
X\subseteq \bigcap_{m^c\in \Sigma} [m^c](V).
\end{equation}
We show that, if this is the case, then the first conclusion holds. Assuming that $X$ satisfies \eqref{eqqq}, it is enough to prove that the following properties hold true:
\begin{enumerate}
\item $X$ is a special subvariety, say of the form $\alpha+B$, for some torsion point $\alpha\in A$ and $B$ an abelian subvariety of $A$;
\item The order of $X$ is bounded by $\ord(X)\leq n^c=d^{o(1)}$ (and the same applies for the order of $l\alpha+B$, for $l\in L$)\footnote{The estimate $n^c=d^{o(1)}$ follows from the more precise \eqref{eqonprimes}.};
\item For every $l\in L$, $l\alpha+B$ is a component of $V$.
\end{enumerate}
Letting $Z$ be the union of such components provides the first conclusion.
\begin{proof}[Proof of (1).]
From the set $\Sigma$, defined in \eqref{setS}, pick
\begin{displaymath}
q_1=(p_1)^c< q_2=(p_2)^c < \ldots< q_{D+1}=(p_{D+1})^c,
\end{displaymath}
such that
\begin{equation}\label{prime range bound}
 \sqrt{D}\leq q_1 < q_{D+1}\leq n^c,
\end{equation}
and $p_1,\ldots,p_{D+1}$ relatively prime to each other. For simplicity we can find these $p_i$'s among prime numbers, which are known to abound. We can actually choose the $q_i$'s among
\begin{displaymath}
\Sigma \cap \{p(\left\lceil{\sqrt[2c]{D}}\right \rceil)^c,\ldots,p(\left\lceil{\sqrt[2c]{D}}\right\rceil +D+\omega(d)+1)^c\},
\end{displaymath}
to take care that we might have at most $\omega(d)$ primes dividing $d$, and to exclude at most one more prime, accounting for the characteristic $p$ of the base field $k$. The lower bound in \eqref{prime range bound} then holds, as well as the upper bound, as $n$ was defined for this very purpose (the choice of $N$ will be clear later).

For every $q\in\{q_1,\ldots , q_{D+1}\}$ (or more generally in the set $\Sigma$) by assumption \eqref{eqqq}, we have
\begin{displaymath}
X\subseteq [q](V),
\end{displaymath}
and so there exists at least a component $Y(q)$ of $V$ such that
\begin{displaymath}
X\subseteq [q](Y(q)).
\end{displaymath}
Since $X$ is a component of maximal dimension, it follows that $X=[q](Y(q))$. As $q$ ranges through $\{q_1,\dots, q_{D+1}\}$ the numbers of distinct $Y(q_i)$ cannot be more than $\deg(V)=D$. That is, a collision must occur. We have
\begin{displaymath}
Y(q_{i_1})=Y(q_{i_2}),
\end{displaymath}
with $1\leq i_1\neq i_2\leq D+1$ (this is why we have chosen to work with $D+1$ numbers). Write $Y:=Y(q_{i_1})=Y(q_{i_2})$, we have
\begin{displaymath}
\deg(Y)\leq \deg(V)=D\qquad\text{and} \qquad [q_{i_1}](Y)=X=[q_{i_2}](Y).
\end{displaymath}
As we guaranteed that $D< (q_{i_1} q_{i_2})^2$, Proposition \ref{Hindry Lemma Criterion} implies that $Y$ is special. Hence so is $X$.
\end{proof}
This covers the first claim, and we are now ready to bound $\ord(X)$. 
\begin{proof}[Proof of (2)] Thanks to (1), we can find an abelian subvariety $B\subset A$ such that $Y=\alpha_Y+B,$
for some torsion point $\alpha_Y\in A$, and we may assume that $\ord (Y) =\ord (\alpha_Y)$ by the general assumptions. We then have
\begin{displaymath}
X=q_{i_1}\cdot \alpha_Y+B = q_{i_2}\cdot \alpha_Y+B.
\end{displaymath}
Necessarily we have the divisibility
\begin{displaymath}
\ord(Y)|\left|{q_{i_1}-q_{i_2}}\right|\leq n^c.
\end{displaymath}
On the other hand we know, for some $l\in L$, that $l a +A' \subset Y$, and so 
\begin{displaymath}
\ord(Y)|\ord(l\cdot a+A')=\ord(a+A')=d.
\end{displaymath}
Hence each $q$ in the set $\Sigma$ is relatively prime to $\ord(Y)$ (sine it is relatively prime to $d$). We deduce that $\ord(X)=\ord(Y) \leq n^c,$
concluding the proof of (2). 
\end{proof}
Set $d':=\ord(X)$. We are left to prove that for every $l\in L$, $l\alpha+B$ is a component of $V$. Namely, for any $1\leq j<d'$ such that $j \wedge d'=1$, we have
\begin{displaymath}
[j^c](X)\subseteq V.
\end{displaymath}
To take care of the fact that such a $j^c$, a priori, does not lie in the set $\Sigma$ (since $j\wedge d$ could be different from $1$), we use some results on the Jacobsthal's function $g(d)$, as discussed in Section \ref{proofmainprop}. It is only in this proof that we need to argue with $N$, rather than $n$.
\begin{proof}[Proof of (3)]
We choose an inverse $j^\ast+d' \Z$ of $j+d'\Z$, represented by some $1\leq j^\ast<d'$. By Proposition \ref{jacob} there is some $m$ such that $m\equiv j^\ast\pmod{d'}$, relatively prime to $d$ and not divisible by $p$, satisfying
\begin{displaymath}
m\leq d'\cdot g ((d \cdot \max \{1,p\})/d')\leq n^c\cdot g(d \cdot \max \{1,p\})=N,
\end{displaymath}
since, by (2), $d'\leq n^c$ (this explains our definition of $N$). Therefore $m^c$ belongs to the set $\Sigma= \Sigma (d,D)$, defined in \eqref{setS}. In particular we have
\begin{displaymath}
X=[m^c](Y(m^c)),
\end{displaymath}
for the component $Y=Y(m^c)$ of $V$. We know, for some $l\in L$, that $a_Y:=l\cdot a$ belongs to $Y$ and we can write $Y=a_Y+B,$
for some abelian subvariety $B$ of $A$. Therefore
\begin{displaymath}
X=[m^c](Y)=m^c\cdot a_Y+B.
\end{displaymath}
Applying $[j^c]$, the above identity becomes
\begin{displaymath}
[j^c](X)=[(jm)^c](a_Y+B)=(jm)^c\cdot a_Y+B.
\end{displaymath}
But $\ord(a_Y+B)=d'$ and $jm\equiv 1\pmod{d'}$, so $[j^c](X)=Y$
is indeed a component of $V$.
\end{proof}
The proof of Lemma \ref{lemma1} is completed.
\end{proof}

\subsubsection{Second lemma}
We refine the conclusion of Lemma \ref{lemma1}, constructing a special subvariety in $V$ of explicitly bounded degree.

\begin{lemma}\label{lemma2}
Under the same assumptions and notations of Proposition \ref{keyprop}, there is $ Z$ such that
\begin{displaymath}
L\cdot a+A' \subseteq Z\subseteq V,
\end{displaymath}
and of the form $L\cdot \alpha+B$, for some abelian subvariety $B\subset A$, and $\alpha \in A_{\operatorname{tors}}$, such that
\begin{displaymath}
\deg(Z)\leq D^{O(1)}\cdot d^{o(1)}.
\end{displaymath}
\end{lemma}
The proof of the above result is obtained by induction from Lemma \ref{lemma1}. Notice that, for every $l \in L$, we have
\begin{displaymath}
\ord (l\alpha + B)\leq \deg (Z)\leq D^{O(1)}\cdot d^{o(1)},
\end{displaymath}
which is exactly what we need. During the proof we will actually prove that
\begin{displaymath}
\deg(Z)\leq  \max\{D, (\omega(d)+1)\}^{\delta(\Delta,c)} \left(2\cdot g(d\cdot \max\{1,p\})\right)^{\delta'(\Delta,c)},
\end{displaymath}
where $\delta(\Delta,c)$ and $\delta'(\Delta,c)$ are explicit constant depending only on $\Delta$ and $c$ (see \eqref{estimate D prime} for their definition).

\begin{proof}
We define inductively
\begin{displaymath}
f_0(D)=D,\qquad f_{i+1}(D)=f(f_i(D),d),
\end{displaymath}
where $f(D,d)= D^2\cdot \left(n^c\cdot g(d \cdot \max\{1,p\})\right)^{2\cdot c\cdot\Delta}$ was defined in \eqref{def f D d}, for $n=p(x)$, and $x=x_D:=2D+\omega(d)+1$. For simplicity, from now on, we ignore the factor $\max\{1,p\}$ in the argument of the Jacobsthal's function. For example, we have
\begin{displaymath}
f_2(D,d)=(D^2n^{2c^2\Delta}g(d)^{2c\Delta})^2({n_2}^cg(d))^{2c\Delta}=D^2n^{2c^2\Delta} {n_2}^{2c^2\Delta}g(d)^{6c\Delta},
\end{displaymath}
where $n_2=p(x_2)$, and 
\begin{displaymath}
x_2:=x_{f(D,d)}=2f(D,d)+\omega(d)+1= 2D^2\cdot \left(n^c\cdot g(d)\right)^{2\cdot c\cdot\Delta}+\omega(d)+1.
\end{displaymath}
Set also
\begin{displaymath}
V_0=V,\qquad V_{i+1}={V_i}',
\end{displaymath}
with ${V_i}'$ to be $V'$ constructed from Lemma \ref{lemma1} applied to the $V_i$. The latter is defined until $i_0$, for some $i_0$, where the first point of Lemma \ref{lemma1} occurs, and in this case we obtain $Z$ with $\deg(Z)\leq \deg(V_{i_0})$ and of the form $L\cdot \alpha+B$,
for some abelian subvariety $B\subset A$, and $\alpha \in A_{\operatorname{tors}}$. This has to happen for some $i_0\leq \Delta$, since at each step the dimension is dropping by one, and each $V_i$ is not empty.

The functions $f_i(D)$ are increasing in both $D$ and $i$, and so we can carry over upper bounds and rely on the computation below. For example, for any $i$, we get
\begin{displaymath}
\deg(V_i)\leq f_i(D)\leq  D':=f_\Delta(D),
\end{displaymath}
which will be bounded in an explicit way in \eqref{estimate D prime} below.

Set $D_\ast:=\max\{D, \omega(d)+1\}$, and notice that
\begin{displaymath}
4 \leq x=2D+\omega(d)+1\leq 3 D_\ast.
\end{displaymath}
In particular we can write
\begin{displaymath}
n=p(x)\leq x\log(x)(1+\log \log (x))\leq 3D_\ast\log(3D_\ast)(1+\log \log (3D_\ast)) \leq {D_\ast}^{1+\epsilon},
\end{displaymath}
for some $\epsilon>0$, and 
\begin{displaymath}
f_1(D)=f(D,d)= D^2\cdot n^{2c^2 \Delta} \cdot g(d)^{2\cdot c\cdot\Delta}\leq {D_\ast}^{2(1+c^2\cdot \Delta\cdot (1+\epsilon))}\cdot g(d)^{2\cdot c\cdot \Delta}.
\end{displaymath}
To bound the $f_i(D)$, for $1<i\leq \Delta$ we argue inductively. A direct computation shows that
\begin{align}\label{estimate D prime}
f_\Delta(D)=D'&\leq {D_\ast}^{{2^\Delta}(1+\lambda)}\cdot (2\cdot  g(d))^{2^{(\Delta+1)}c\Delta}&\text{ with }&\lambda=c^2\cdot \frac{\Delta^2-\Delta}{2}\cdot(1+ \epsilon ),
\\
&=
 \max\{D, (\omega(d)+1)\}^{\delta } \cdot (2\cdot g(d))^{\delta '}
&\text{ with }&
\left\{
\begin{aligned}
\delta=\delta(\Delta,c)&={2^\Delta}(1+\lambda);\\
\delta'=\delta'(\Delta,c)&=2^{\Delta}\lambda/c.
\end{aligned}
\right.
\end{align}
where $\epsilon=0.5$ will always do. We have therefore proven the desired bound
\begin{displaymath}
\deg(Z)\leq D^{O(1)}\cdot d^{o(1)}.
\end{displaymath}

\end{proof}

%We are ready to prove Proposition \ref{keyprop}, as was stated in Section \ref{statementprop}. It consists in a refinement of the conclusion of Lemma \ref{lemma2}.
\subsubsection{Proof of Proposition \ref{keyprop}}
Let $V$ be a subvariety of $A$ of degree $D$ and dimension $\Delta<g$, and $a\in \Ators$, of order $d$, such that $L\cdot a\subseteq V$, where $L=L(c)$ is the Lang group as in \eqref{langdef}. Consider the set $\Theta$ of subvarieties of $V$ of the form
\begin{displaymath}
Y=L\cdot a_Y+B_Y,
\end{displaymath}
where $a_Y\in \Ators$ and $B_Y$ is an abelian subvariety of $A$, such that $L\cdot a \subseteq Y.$
Notice that $a\in \Theta$, and in particular $\Theta\neq \emptyset$. We can choose a $Y_{\min} = L\cdot {a}_{\min} + B_{\min} $ such that the order of the coset $a_{\min}+B_{\min}$ is minimal among the elements in $\Theta$. Lemma \ref{lemma2}, applied to the pair $(a_{\min}+B_{\min}\in \Ators, V)$ asserts that there exists $Z=L \alpha_Z + B_Z$, containing $Y_{\min}$, and of order bounded by $d^{o(1)}$, and implicit constant depending only on $D, \Delta,c$\footnote{Here is where we needed a version of Lemma \ref{lemma2} for torsion cosets, rather than for just torsion points.}. Let $\delta=\delta(D, \Delta,c)$ be the biggest number such that $d^{o(1)}<d$, for all $d \geq \delta$. My minimality of the order of $a_{\min}$, it follows that
\begin{displaymath}
\ord (a_{\min})\leq \min \{\delta,d\}
\end{displaymath}
and therefore $Y_{\min}= L\cdot {a}_{\min} + B_{\min}$ satisfies the desired conclusion. Proposition \ref{keyprop} is proven.

\subsection{Proof of Theorem \ref{thm22}, and Theorem \ref{mainthm}}\label{proofhor}
We are ready to prove Theorem \ref{thm22}, which, as explained in Section \ref{langsection}, implies Theorem \ref{mainthm}. Let $\mathcal{A}$ be an abelian scheme over $\Oo_{K,S}$. %Our proof works both when the special 0-cycles are of characteristic zero and when they are of arbitrary characteristics. Here $A=A_K= \mathcal{A}\times_{\Oo_{K,S}}K$, and the implicit constant $c=c(A)$ in the Lang group $L=L(c)$ is the one given by Theorem \ref{lang} applied to $A$.

\begin{proof}[Proof of horizontal part of Theorem \ref{thm22}]
Let $(E_n)_{n\geq 0} \subset \mathcal{A}$ be a sequence of special 0-cycles associated to $a_n \in A_{\operatorname{tors}}$ and $\mathfrak{p}_n\in \Spec (\Oo_{K,S})$. Consider the image of the $a_n$ along the corresponding map 
\begin{displaymath}
\pi_n: A(\overline{K})\to \mathcal{A}_{\mathfrak{p}_n}(\overline{\kappa(\mathfrak{p}_n)}).
\end{displaymath}
Denote by $d_n$ the order of $a_n\in A$, and by $d_n'$ the order of $\pi_n(a_n)$ in $\mathcal{A}(\overline{\kappa(\mathfrak{p}_n}))$. 

We study $\mathcal{V}$, the Zariski closure of $\bigcup_{n\geq 0} E_n$ in $\mathcal{A}$. It is a finite union of vertical fibres and horizontal components (i.e. flat over the base  $\operatorname{Spec}(\Oo_{K,S})$). We may assume that $d_n$ and $d_n '$ tend to infinity, as $n$ grows, since otherwise $\mathcal{V}$ is special (by definition). Notice that, for any $n\geq 0$, we have
\begin{displaymath}
L\cdot \pi_n(a_n)=\pi_n(L\cdot a_n)\subseteq E_n.
\end{displaymath}

We first look at the horizontal/flat part, even if the same argument applies to the vertical case, so, up to extracting a subsequence, we may assume that $|\mathfrak{p}_n| \to + \infty$. Set $D$ for the maximum, over $n$, of $\deg (V_{\mathfrak{p}_n})$, $\Delta:=\dim (V_K)$, and let $\delta$ for the maximum of $\delta(\deg (V_{\mathfrak{p}_n}),\Delta,c)$ be as in Proposition \ref{keyprop} applied to the pair $\pi_n(a_n)$ and $V_{\mathfrak{p}_n}\subset \mathcal{A}_{\mathfrak{p}_n}$ (it is a finite set, since $\deg (V_{\mathfrak{p}_n})$ takes only finitely many values as $n$ varies). We can find a positive integer $q$ such that
\begin{itemize}
\item $q$ is a $c$-th power;
\item $q$ is relatively prime to $\delta $;
\item $q$ is relatively prime to $p_n:=\operatorname{char}(\kappa (\mathfrak{p}_n))$, for all $n$ bigger than a given constant $M$.
\end{itemize}
Proposition \ref{keyprop} asserts that, for each torsion point $\pi _n (a_n) \in \mathcal{A}_{\mathfrak{p}_n}$, we have
\begin{displaymath}
L\cdot \pi _n (a_n) \subseteq \left( L\cdot\alpha_n+B_n \right) \subseteq  \mathcal{V},
\end{displaymath}
for some $\alpha_n \in \mathcal{A}_{\mathfrak{p}_n}$, with $\ord(\alpha_n + B_n)\leq \delta$ (in particular the order of $\alpha_n$ does not depend on $d_n'$, nor on $p_n$), and $B_n \subset  \mathcal{A}_{\mathfrak{p}_n}$ an abelian subvariety. Notice that 
\begin{displaymath}
L\cdot\alpha_n+B_n
\end{displaymath}
is stable under the multiplication by $[q]$ (by definition of $L$). In particular $\mathcal{V}$ can be written as the Zariski closure of the set
\begin{displaymath}
\left( \bigcup_{n\geq M} L\cdot\alpha_n+B_n \right)\subseteq \mathcal{V}.
\end{displaymath}
As explained above, the left hand side is stable under multiplication by $[q]$, and so $\mathcal{V}$ itself is stable under $[q]$. Corollary \ref{corhindry}, applied to $\mathcal{V}_K\subseteq A$, shows that $\mathcal{V}_K$ is special. This means, following the definitions introduced in Section \ref{mainresultsec}, that the horizontal part of $\mathcal{V}$ is a finite union of flat cycles.
\end{proof}
To conclude the proof of Theorem \ref{thm22}, we have to show that the vertical part of $\mathcal{V}$ is a finite union of torsion cycles, essentially by adapting the argument just presented. The lifting argument is more delicate, and presents some challenges depending on the order of the torsion points $a_n$, this will be dealt with in the next section.

\begin{proof}[Proof of the Weak vertical part of Theorem \ref{thm22}]
We are left to consider the case where $(E_n)_{n\geq 0} \subset \mathcal{A}$
is a sequence of special 0-cycles associated to $(a_n)_n \subset A_{\operatorname{tors}}$ and a fixed $\mathfrak{p}\in \Spec (\Oo_{K,S})$. Consider the image of the $a_n$ along the corresponding map $\pi_{\mathfrak{p}}: A(\overline{K})\to \mathcal{A}_\mathfrak{p}(\overline{\kappa(\mathfrak{p})}).$
 Let $d_n'$ be the order of $\pi_{\mathfrak{p}}(a_n)\in \mathcal{A}_\mathfrak{p}(\overline{\kappa(\mathfrak{p})})$. If the $d_n'$ are bounded as $n$ goes to infinity, then the set $\bigcup_{n\geq 0} E_n$
is finite and so there is nothing to prove. We may therefore assume that the $d_n'$ are unbounded.

The argument described above implies that every component of $\mathcal{V}$ is a translate of an abelian subvariety of $\mathcal{A}_\mathfrak{p}$. Indeed, let $c$ be the constant associated to $A$ by Theorem \ref{lang}, and let $\delta= \delta(\deg (\mathcal{V}),\dim (\mathcal{V}) ,c)$ be the constant given by Proposition \ref{keyprop} applied to $\mathcal{V} \subset \mathcal{A}_\mathfrak{p}$ (and any $\pi_{\mathfrak{p}}(a_n)$). We can find a positive integer $q$ such that $q$ is a $c$-th power, and $q\wedge \delta  p=1$. Proposition \ref{keyprop} asserts that, for each torsion point $\pi_{\mathfrak{p}}(a_n) \in \mathcal{A}_\mathfrak{p}$, we have
\begin{displaymath}
L\cdot \pi_{\mathfrak{p}}(a_n)\subseteq \pi_{\mathfrak{p}} (L \cdot a_n) \subseteq \left( L\cdot\alpha+\mathfrak{B}_n \right) \subseteq  \mathcal{V},
\end{displaymath}
for some $\alpha \in \mathcal{A}_\mathfrak{p}(\overline{\kappa(\mathfrak{p})})$, with $\ord(\alpha + \mathfrak{B}_n)\leq \delta$, for some $\delta$ depending only on $D,\Delta$ and $c$ (where $\mathfrak{B}_n$ is some abelian subvariety of $\mathcal{A}_\mathfrak{p}$). Notice that $L\cdot\alpha+\mathfrak{B}_n$ is stable under the multiplication by $[q]$. Since $\mathcal{V}$ is defined as the Zariski closure of the $E_n$, and from the fact that $ \delta$ does not depend on $d_n'$, it follows that $\mathcal{V}$ is stable under the multiplication by $[q]$. Proposition \ref{Hindry Lemma Criterion} then shows that $\mathcal{V}$ is special, that is a finite union of translates by a (torsion) point of $\mathcal{A}_\mathfrak{p}(\overline{\kappa(\mathfrak{p})})$ of abelian subvarieties of $ \mathcal{A}_\mathfrak{p}$. The proof of Theorems \ref{mainthm} and \ref{thm22} is completed.
\end{proof}
The last statement will be referred to as the \emph{weak vertical part} of the arithmetic Manin--Mumford. The \emph{stronger} part, that is the lift of the torsion cosets of $\mathcal{A}_\mathfrak{p}$ to $K$, possibly up to $p$-rank zero factors, is described in the next two sections.

\section{Strong vertical, Theorem \ref{mainthmlift}}\label{liftingcycles}
In this section we prove Theorem \ref{mainthmlift}, which may be of independent interest from the arguments presented so far. As in Section \ref{mainresultsec}, $K$ denotes a number field and $S$ a finite set of finite places of $K$. To explain the method of the proof we first prove the following easier statement, which appeared in the introduction as Theorem \ref{lastthm}. It can be thought as the analogue of the main theorem of Edixhoven and Richard \cite{basrod} for abelian varieties (rather than $\A^1\times_{\Z} \A^1$).
\begin{thm}\label{proplift}
Fix a prime $\mathfrak{p}\subset \Oo_{K,S}$ of residue characteristic $p>0$. Let $\mathcal{A}$ be an abelian scheme over $\Oo_{K,S}$, $(a_n)_n\subset A_{\operatorname{tors}}$ be a sequence of torsion points of order coprime with $p$ and write $E_n=\pi_{\mathfrak{p}}(\Gal(\overline{K}/K) \cdot a_n)$, where $
\pi_{\mathfrak{p}}: A(\overline{K})\to \mathcal{A}_\mathfrak{p}(\overline{\kappa(\mathfrak{p})})$.
Then every component of the Zariski closure $\mathcal{V}$ of 
\begin{displaymath}
\bigcup_{n\geq 0} E_n\subseteq \mathcal{A}_\mathfrak{p}\subset \mathcal{A},
\end{displaymath}
is a component of the intersection between a flat cycle $\mathcal{B}\subseteq \mathcal{A}$ and $\mathcal{A}_\mathfrak{p}$.
\end{thm} 
The novelty of Theorem \ref{proplift}, compared to Theorem \ref{mainthm}, is that the components of $\mathcal{V}$ are not only given by torsion cycles of $\mathcal{A}_\mathfrak{p}$, but they actually come from characteristic zero. The case where torsion points are allowed to have order divisible by $p$ will be addressed in Section \ref{padicsection}. From now on, we let $p$ be the characteristic of the residue field $\kappa(\mathfrak{p})$.

Before proving Theorem \ref{proplift}, we explain how it implies Theorem \ref{mainthmlift}.
\subsubsection{Proof of Theorem \ref{mainthmlift}}\label{conclusion}
Assuming Theorem \ref{proplift}, which will be proven in Section \ref{proofthm40111}, we show how to conclude the proof of the more general Theorem \ref{mainthmlift}, as stated in the introduction. Let $E_n$ be a sequence of special zero cycles in $\mathcal{A}$, each one associated to $a_n\in \Ators$ and primes $\mathfrak{p}_n\subset \Oo_{K,S}$.

Let $\mathcal{V}\subseteq \mathcal{A}$ be the Zariski closure of the $(E_n)_n$,
and write it as a finite union of flat and vertical components. From Theorem \ref{mainthm}, which was proven in Section \ref{proofhor}, we know that the flat part is a finite union of flat cycles, and that each vertical component at $\mathfrak{p}_1,\dots , \mathfrak{p}_m \subset \Oo_{K,S}$ is a torsion cycle of some $A_{\mathfrak{p}_j}$ for $j\in \{1,\dots ,m\}$. Theorem \ref{proplift}, applied at $\mathfrak{p}_1,\dots, \mathfrak{p}_m$, now implies that each torsion cycle is given as the intersection between a flat cycle and vertical fibres. That is $\mathcal{V}$  can be written as a finite union of flat cycles and intersections between flat cycles and vertical fibres. This concludes the proof of Theorem \ref{mainthmlift}.

\subsection{Variants of the Tate conjecture (after Faltings et. al.)}\label{sectiontate}
We recall Faltings's results about the \emph{semi-simplicity of the Tate module} and the \emph{Isogeny Theorem} (or often referred to as the \emph{Tate conjecture}) and some variants thereof. 
\begin{thm}\label{tatethm}
Let $K$ be a number field, and $A$ an abelian variety over $K$. 
\begin{itemize}
\item (Faltings \cite{faltings}.) Let $\ell$ be a rational prime, and $V_\ell(A)$ the rational $\ell$-adic Tate module of $A$. The canonical map
\begin{equation}\label{tateconj1}
\End_K(A)\otimes_{\Z} \Q_\ell \to \End_{\Gal(\overline{K}/K)}(V_\ell(A))
\end{equation}

is an isomorphism. 
\item (Faltings, Zarhin  \cite[Sec. 5.4]{zarhin}.) For any $\ell$ large enough, the canonical map
\begin{equation}\label{tateconj}
\End_K(A)\otimes \Z/\ell\Z\to \End_{\Gal(\overline{K}/K)}(A[\ell])
\end{equation}
is an isomorphism. 
\item (Masser-W\"{u}stholz \cite{mw, mwisogeny}.) For any positive integer $m$, the cokernel of the canonical map 
\begin{equation}\label{mwthm}
\End_K(A) \to \End_{\Gal(\overline{K}/K)}(A[m])
\end{equation}
is killed by $M$, for some positive integer $M\leq c \left(\max\{1, h(A)\} \right)^\tau$, where $\tau$ is a constant depending only on $[K:\Q]$, and $c$ depends only on $\dim A$ and $[K:\Q]$.
\item (\cite[Thm. 4.2, Ch. IV (page 147)]{zbMATH00047975}) Let $(\ell_n)_{n \in \N}$ be a sequence of distinct prime numbers and let $R= \prod_{n\in \N} \Ff_{\ell_n}$. Since $R\otimes_{\Z} \Q$ is a non-zero\footnote{The kernel of the map $R\to R\otimes_{\Z} \Q$ is the set of elements $r\in R$ such that $M \cdot r=0$ for some $M \in \Z_{>1}$. Let $r=1\in R$, and $M$ arbitrary. There exists $\ell_n \nmid M$, and, as $M \cdot 1\neq 0$ in $\Ff_{\ell_n}$, we have $M \cdot r \neq 0 \in R$. Thus $r$ has non zero image in $R\otimes \Q$.} $\Q$-algebra, there exists a field $E$ (of characteristic zero) and a ring map $R\otimes_{\Z} \Q\to E$ corresponding to some maximal ideal. Let $V_E(A):= T (A)\otimes _{\widehat{\Z}} R \otimes_{\Q} E$ be the $E$-Tate module of $A$. The canonical map
\begin{equation}\label{etate}
\End_K(A)\otimes_{\Z} E \to \End_{\Gal(\overline{K}/K)}(V_E(A))
\end{equation}
is an isomorphism.
\end{itemize}
Moreover $V_\ell(A),A[\ell], V_E(A)$ are semisimple representation of $\Gal(\overline{K}/K)$ (the second one, for $\ell$ large enough). 
\end{thm}

\subsection{Preliminary results}\label{keypropo}
\subsubsection{First formal descent lemma}
We start with a formal descent lemma for idempotents in scalar extensions of semisimple algebras.
\begin{lemma}\label{keylemma}Let $M, N$ be semisimple (associative, non necessarily commutative, finite dimensional) $\Q$-algebras, with $M$ a subalgebra of $N$. Let $E/\Q$ be a field extension, and
\begin{displaymath}
i_E: N\otimes_\Q E\hookrightarrow \End_E(V)
\end{displaymath}
be a faithful representation on $V$, a finite dimensional $E$-linear vector space. If $u$ is an idempotent of $N$, and $w\in M\otimes_\Q E$ is such that $i_E(w)\cdot V\subseteq i_E(u)\cdot V$, then there exists an idempotent $v\in M$ satisfying
\begin{displaymath}
i_E(w)\cdot V\subseteq i_E(v)\cdot V\subseteq i_E(u)\cdot V.
\end{displaymath}
\end{lemma}
In the proof  of Lemma \ref{keylemma} we use the following standard fact about semisimple $\Q$-algebras. Let $M$ be a semisimple associative finite dimensional $\Q$-algebra, then any $\Q$-linear subspace $X$ of $M$ is a right ideal if and only if it is of the form $vM$ for some idempotent $v$ of $M$. % (possibly, of course, $v\in \{0,1\}$).
\begin{proof}[Proof of Lemma \ref{keylemma}]
Consider $Y=uN$, as a right $N$ linear submodule of $N$. Let $X:=Y\cap M$ be the associated right $M$-module. By the previous fact, there is at least one idempotent $v\in M$ with the following property: for any $m$ in $M$ we have $u\cdot m=m$ if and only if $v\cdot m=m$. Otherwise stated, we have $X=v\cdot M$. The same property holds true with elements in $M\otimes E$. It is indeed enough to notice, just by computing dimensions, that we have $(Y\otimes_\Q E) \cap (M\otimes_\Q E)= X \otimes_\Q E$.

As $\im (i_E(w) )$ is a submodule of $\im (i_E(u))$, by assumption, and $i_E(u)$ is an idempotent, we conclude that $i_E(u \cdot w)=i_E(w)$. By faithfulness of $i_E$, we have also $u\cdot w = w \in N\otimes_\Q E$.

Let $v\in M$ be the idempotent associated to $u \in N$ as above. It follows that 
\begin{displaymath}
i_E(v) \cdot i_E(w) =i_E(w),
\end{displaymath}
that is
\begin{displaymath}
\im (i_E (w))\subseteq \im(i_E (v)).
\end{displaymath}
 Finally, as $v\in vM\subset u N$, we get $u\cdot v=v$, and likewise 
\begin{displaymath}
\im (i_E(v))\subseteq \im(i_E(u)).
\end{displaymath}
In another writing, this is indeed
\begin{displaymath}
i_E(w)\cdot V\subseteq i_E(v)\cdot V\subseteq i_E(u)\cdot V.
\end{displaymath}
\end{proof}

\subsubsection{Applications}
We first present the simplest result we use to lift torsion cycles, it encompasses the power of \eqref{tateconj1}, making it ready to be applied for our purposes. Let $A$ be an abelian variety over $K$, $\mathcal{A}_\mathfrak{p}$ its reduction mod $\mathfrak{p}$, a place of good reduction. Set $p=\operatorname{char} (\kappa(\mathfrak{p}))>0$, and
\begin{displaymath}
M^0:=\End_K (A)\otimes_\Z \Q, \ \ \ N^0:=\End_{\kappa (\mathfrak{p})}(\mathcal{A}_\mathfrak{p})\otimes_{\Z} \Q.
\end{displaymath}
The reduction mod $\mathfrak{p}$ defines a canonical embedding $M^0\hookrightarrow  N^0,$
and, for every $\ell \neq p$, we have
\begin{displaymath}
M^0\otimes \Q_\ell \hookrightarrow N^0 \otimes \Q_\ell \hookrightarrow \End_{\Q_\ell} (V_\ell (A)),
\end{displaymath}
thanks to the identification $V_\ell(A)\cong V_\ell (\mathcal{A}_\mathfrak{p})$ given again by the reduction mod $\mathfrak{p}$.

\begin{prop}\label{lemmatolift}
 %Let $A$ be an abelian $K$-variety, $\mathcal{A}_\mathfrak{p}$ its reduction mod $\mathfrak{p}$, a place of good reduction. Let $\ell$ be a rational prime different from $p$. 
 Let $W$ be a $\Gal(\overline{K}/K)$-invariant subspace of $V_\ell(A)$ and let $\mathfrak{B}\subseteq \mathcal{A}_\mathfrak{p}$ be an abelian subvariety such that
 \begin{displaymath}
 W \subseteq  V_\ell (\mathfrak{B}) \subseteq V_\ell (A). 
 \end{displaymath}
 There exists an abelian subvariety $D$ of $A$, defined over $K$, such that
 \begin{displaymath}
W \subseteq V_\ell (D)\subseteq  V_\ell (\mathfrak{B}).
\end{displaymath} 
\end{prop}
Notice that Proposition \ref{lemmatolift} can be rewritten in purely algebraic therms. Here $M^0$ and $N^0$ are finite dimensional semisimple $\Q$-algebras. Moreover $ \mathfrak{B}\subseteq \mathcal{A}_\mathfrak{p}$ corresponds to an idempotent $u=u^2 \in N^0$ such that
\begin{displaymath}
W\subseteq u\cdot V_\ell(A),
\end{displaymath}
and Proposition \ref{lemmatolift} asserts that there exists an idempotent  $v$ in $M^0$ such that
\begin{displaymath}
W\subseteq v\cdot V_\ell(A)\subseteq u\cdot V_\ell(A).
\end{displaymath}

\begin{proof}[Proof of Proposition \ref{lemmatolift}]Since $V_\ell (A)$ is a semisimple Galois representation, as explained at the end of Theorem \ref{tatethm}, we have that $W \subseteq V_\ell(A)$ is of the form $e \cdot V_\ell (A)$ for some idempotent $e \in \End_{\Gal (\overline{K}/K)}(V_\ell(A)).$ By the Tate conjecture (\eqref{tateconj1} in Theorem \ref{tatethm}) we can actually write
\begin{displaymath}
W = w \cdot V_\ell (A)
\end{displaymath}
for some $w\in M^0\otimes_\Q\Q_\ell$ (and such $w$ is unique by faithfulness). Lemma \ref{keylemma}, applied with $E=\Q_\ell$, shows how to get $v$ in $M^0$ (rather than in $M^0\otimes_\Q \Q_\ell$) such that the conclusion holds.
\end{proof}

Replacing in the proof of Proposition \ref{lemmatolift} the role of the $\ell$-adic Tate conjecture by the \emph{adelic Tate conjecture} \eqref{etate}, we obtain:
\begin{prop}\label{lemmatoliftadelic}
% Let $A$ be an abelian variety over $K$, $\mathcal{A}_\mathfrak{p}$ its reduction mod $\mathfrak{p}$, a place of good reduction.
 Let $(\ell_n)_{n\geq 0}$ be a sequence of distinct primes different from $p$, and let $E$ be the field associated to $\Q \otimes \prod_n \Ff_{\ell_n}$ as in \eqref{etate}.  Let $W$ be a $\Gal(\overline{K}/K)$-invariant subspace of $V_E(A)$ and let $\mathfrak{B}\subseteq \mathcal{A}_\mathfrak{p}$ be an abelian subvariety such that
 \begin{displaymath}
 W\subseteq V_E (\mathfrak{B})\subseteq V_E (A).
 \end{displaymath}
 There exists an abelian subvariety $D$ of $A$, defined over $K$, such that
 \begin{displaymath}
W\subseteq V_E(D)\subseteq V_E(\mathfrak{B}).
\end{displaymath} 
\end{prop}

%We record here another application of Lemma \ref{keylemma} which will be useful later.
%\begin{lemma}\label{lemmaextension}
%Let $A$ be an abelian variety over $K$, and $K'/K$ be a finite field extension. Let $E$ be $\Q_\ell$, for some prime $\ell$, (resp. a field as in Proposition \ref{lemmatoliftadelic}) and $W$ be a $\Gal(\overline{K}/K)$-invariant subspace of $V_\ell(A)$ (resp. $V_E(A)$). Let $u \in \End (A_{K'})\otimes \Q$ be an idempotent, and $w$ be an element of $\End (A_K)\otimes_\Q E$ such that $w\cdot V_E(A)=W$. If 
%$W \subseteq u \cdot V_E(A)$, then there exists an idempotent $v\in \End (A_K)\otimes \Q$, such that
%\begin{displaymath}
%w\cdot V_E(A)\subseteq v \cdot V_E(A) \subseteq u\cdot V_E(A).
%\end{displaymath}
%\end{lemma}
%The above result, for example, can be applied to $K'$ a finite field extension such that all endomorphism of $A\times \overline{K}$ are defined over $K'$.

\subsubsection{Combinatorial result}\label{sectiononcomblemma} We describe a combinatorial result that will be used in the proof of Proposition \ref{prop222}.
Let $\ell$ be a prime number, $\Ff:=\Ff_\ell$ the field with $\ell$ elements and $d$ a positive integer.
\begin{prop}\label{Lemme inclusion exclusion} 
Let $\Ff^d$ be a $d$-dimensional $\Ff$-vector space and let $a \in \Ff^d$. Let $G \subset \Gl_d(\Ff)$ be a subgroup. If $V\subseteq \Ff^d$ is a subvector space such that, for some $1 \leq C \leq \infty$,
\begin{displaymath}
\#\{G\cdot a\cap V\}\geq\frac{1}{C}\# \{G\cdot a\},
\end{displaymath}
then there exists a non trivial subvector space $W \subseteq V$ such that 
\begin{equation}\label{borne stabilisateur}
[G:\Stab_G(W)]\leq 3\cdot C^{4^{\dim(V)}}.
\end{equation}
Moreover $G\cdot  a \cap W \neq \emptyset$ and it generates $W$.
\end{prop}

%The next corollary gives a finite index subgroup of $G$ with the desired property in a way that it is independent of $\ell= \text{char} (\Ff)$.
\begin{cor}
With the notation of Proposition \ref{Lemme inclusion exclusion}, in particular $W$ and $V$ contain the orbit $H \cdot g \cdot a$, for some $g\in G$, under the subgroup $H=\Stab _G(W) \subseteq G$, and $[G:H]\leq 3 \cdot C^{4^{\dim V}}$. % Denote by $H'$ the intersection of all subgroups of $G$ of index at most $3\cdot C^{4^{d}}$, we have 
%\begin{displaymath}
%H'\cdot g \cdot  a\subseteq V.
%\end{displaymath}
\end{cor}

The proof is inspired by the arguments of \cite[Sec. 3]{zbMATH05837877}.

\begin{proof}[Proof of Proposition \ref{Lemme inclusion exclusion}]Given a subset $S\subseteq G\cdot a$ and $W\subseteq \Ff^d$, we set
\[
\eps(S): =\frac{\# S}{\#G\cdot a}, \qquad\eps(W):=\eps(G\cdot a\cap W)=\frac{\#\{G\cdot a\cap W\}}{\# G\cdot a}.
\]
Write $\gen{S}=\sum_{s\in S} \Ff \cdot s$ for the subvector space of $\Ff^d$ generated by $S$. Notice that $0\leq \eps(W)\leq 1$.

We argue considering subvector spaces $W\subseteq V$ chosen to be of minimal dimension among all subvector spaces of $V$ and maximizing the quantity
\[
\eps(W)^{4^{(\dim(W)-\dim(V))}}.
\]
Such $W$ exists since there are only finitely many possibilities for $\eps(W)$ and $\dim(W)$. In particular we have 
\begin{equation}\label{majoration eps non nul}
\eps(W)^{4^{(\dim(W)-\dim(V))}}\geq \eps(V)^{4^0}>0,
\end{equation}
and
\begin{equation}\label{majoration eps finale}
{(1/\eps(W))}\leq {(1/\eps(V))}^{4^{(\dim(V)-\dim(W))}}\leq {(1/\eps(V))}^{4^{\dim(V)}}\leq C^{4^{\dim(V)}}.
\end{equation}
Comparing $W$ to all its subvector spaces $W'\subsetneq W$, we also have
\begin{equation}\label{inegalite optimalite}
\eps(W')< \eps(W)^{4^{(\dim(W)-\dim(W'))}}\leq\eps(W)^{4}\leq \eps(W).
\end{equation}

We consider $S:=W\cap (G\cdot a)$, and notice that it generates $W$, since otherwise $W_S=\gen{S}$ would  contradict the inequality \eqref{inegalite optimalite}. The following simple observations will be useful for the sequel:
\begin{itemize}
\item If $gW=W$, then  $gS=g({G}\cdot a\cap W)=g{G}\cdot a\cap gW={G}\cdot a\cap W=S$;
\item Reciprocally, if $gS=S$ then $gW=g\gen{S}=\gen{gS}=\gen{S}=W$.
\end{itemize}
Denote by $G(W)\subseteq G$ the stabiliser of $W$ in $G$ (which is also the stabiliser of $S$). If $gS\neq S$, the vector space $W':=g W \cap W$ satisfies $W'<W$, as well as
\[
gS\cap S\subseteq gW\cap W<W.
\]
The inequality \eqref{inegalite optimalite} then gives
\[\frac{\#gS\cap S}{\#G\cdot a}
\leq
\frac{\#(G\cdot a\cap W')}{\#G\cdot a}
=
\eps(W')< \eps(W)^4.\]
More generally, given $g\cdot G(W)\neq g'\cdot G(W)$, we similarly find
\begin{equation}\label{borne intersection induction eps}
\frac{\#gS\cap g'S}{\#G\cdot a}\leq \frac{\# g'^{-1}gS\cap S}{\#G\cdot a}< {\eps(W)}^4.
\end{equation}
To ease the notation, from now on, we simply set $\eps=\eps(W)$. Let $c=[G:G(W)]$ and choose some representative $g_1\cdot G(W),\ldots, g_c\cdot G(W)$ of $G/G(W)$. We are left to prove that $c\leq 3\cdot 1/\eps$. 

\begin{proof}[Proof of the bound for $c$]
If $c>3\cdot 1/\eps$, we set $\gamma=\lceil1/\eps+1\rceil$. Recall that $0<\eps\leq 1$, and the evident bounds $1/\eps+1 \leq \gamma<1/{\eps}+2$ and $\gamma\leq c$
(since $2\leq 2\cdot (1/\eps)$). We can also notice that
\begin{align}
\gamma\cdot \eps&\geq (1/{\eps}+1)\cdot {\eps}=1+{\eps}\\
\frac{\gamma(\gamma-1)}{2}\cdot {\eps}^4&<\frac{(1/{\eps}+1)\cdot 1/{\eps}}{2}\cdot {\eps}^4=\frac{{\eps}^2+{\eps}^3}{2}\leq {\eps}^2.
\end{align}
From the above we see that the classes $g_1\cdot G(W),\ldots,g_\gamma\cdot G(W)$ are well defined and all distinct. Thanks to inclusion-exclusion we estimate:
\begin{equation}\label{inclusion exclusion}
\# G\cdot a\geq \#\bigcup_{i=1}^\gamma g_iS\geq \sum_{i=1}^\gamma \#g_iS-\sum_{1\leq i<j\leq \gamma} \#(g_iS\cap g_j S).
\end{equation}
The number of terms of the last summation is $(\gamma(\gamma-1))/2$, and each term is bounded, via \eqref{borne intersection induction eps}, by ${\eps}^4\cdot \# G\cdot a$. Therefore
\begin{equation}\label{exclusion}
\sum_{1\leq i<j\leq \gamma} \#\{g_iS\cap g_j S\}\leq \frac{\gamma(\gamma-1)}{2}\cdot {\eps}^4\cdot \# G\cdot a
\end{equation}
For the second to last summation in \eqref{inclusion exclusion}, we notice that
\begin{equation}\label{inclusion}
\sum_{i=1}^\gamma \#g_iS=\gamma\cdot \eps\cdot \# G\cdot a\geq (1/\eps+1)\cdot \eps\cdot \# G\cdot a=(1+\eps)\cdot \# G\cdot a.
\end{equation}
Substituting \eqref{inclusion} and \eqref{exclusion} in \eqref{inclusion exclusion}, and simplifying by  $\#G\cdot a> 0$, we obtain the inequality
\[
1>1+\eps-\eps^2,
\]
which is the contradiction we were aiming for, since the right hand side satisfies $1+\eps-\eps^2\geq 1$.\end{proof}
%To conclude the proof of Proposition \ref{Lemme inclusion exclusion}, we are left to show that
%\[
%c=[G:G(W)]\leq 3/\eps.
%\]
%Bounding the right hand side by three times the inequality \eqref{majoration eps finale}, we obtain the desired bound \eqref{borne stabilisateur}. We know that $\eps(W)>0$, thanks  \eqref{majoration eps non nul}, and that there exists an element $g\cdot a$ of $G\cdot a$ in $W$.
This concludes the proof of the Proposition.
 \end{proof}

\subsection{Proof of Theorem \ref{proplift}}\label{proofthm40111}
As we show in Section \ref{prooof}, the proof of Theorem \ref{proplift} follows from Theorem \ref{mainthm} (the weak vertical Manin--Mumford) and the following:
\begin{prop}\label{prop222}
Let $\Oo_{K,S}, \mathfrak{p}, p, \mathcal{A}$ and $\pi_{\mathfrak{p}}$ be as in Theorem \ref{proplift}. 
 For $i=1, \dots , C$, let $\mathfrak{B}_i\subset \mathcal{A}_\mathfrak{p}$ be abelian subvarieties of $\mathcal{A}_\mathfrak{p}$. Let $D_i/K$ be the biggest abelian subvariety of $A$ whose reduction mod $\mathfrak{p}$ is contained in $\mathfrak{B}_i$. Let $(a_n)_{n\geq 0} \subset A_{\operatorname{tors}}$ be a sequence of torsion points of order coprime with $p$ and, for every $n$, let $E_n=\pi_{\mathfrak{p}}(\Gal(\overline{K}/K) \cdot a_n)$. If
\begin{displaymath}
\bigcup_{n} E_n \subset \bigcup_{i=1}^C \beta_i + \mathfrak{B}_i,
\end{displaymath}
then there exists a finite set $F\subset A$ such that, for all $n$,
\begin{displaymath}
a_n \in F +\bigcup_{i=1}^C D_i.
\end{displaymath}

\end{prop}

\begin{proof}We first reduce the proof to the case in which~$I:=\{i\in\{1;\dots ; C\}\,|\,D_{i,\mathfrak{p}} =\mathfrak{B}_i\}$ is~$\emptyset$.

For every~$i=1,\ldots,r$ let $\tilde{\beta}_i \in A_{\operatorname{tors}}$ be a lift of $\beta_i$ and let~$F_i=\Gal(\overline{K}/K)\cap \tilde{\beta}_i$. Let~$n$ be such that 
\begin{equation}\label{redaction:n et i}
\exists i\in I E_n\cap \beta_i+\mathfrak{B}_i\neq \emptyset.
\end{equation}
Then~$\Gal(\overline{K}/K)\cap a_n\cap \tilde{\beta}_i+ D_i\neq \emptyset$. It follows that~$ \Gal(\overline{K}/K)\cap a_n\subseteq F_i+D_i$. This proves the conclusion for an $n$ satisfying~\eqref{redaction:n et i}.

In particular we may assume that for every~$i\in I$ and~$n$, we have~$E_n\cap \beta_i+\mathfrak{B}_i=\emptyset$, and we replace~$\bigcup_{i\in\{1;\ldots;r\}} \beta_i+\mathfrak{B}_i$ by~$\bigcup_{i\in\{1;\ldots;r\}\smallsetminus I} \beta_i+\mathfrak{B}_i$.

We are reduced to the case~$I=\emptyset$, and now proceed by reducing the proof to the case where ~$\beta_1=\ldots=\beta_r=0$.

Let~$m_i:= \ord (\beta_i)$ and set~$m:=m_1\cdot\ldots\cdot m_r$. The order of~$a'_n=m\cdot a_n$ is coprime to~$p$
and the corresponding~$E'_n$ satisfy~$\bigcup_n E'_n=m\cdot \bigcup_n E_n\subseteq \bigcup_{i=1}^r \mathfrak{B}_i$. Assume that, for some finite set~$F'$, we have~$a'_n\in F'+\bigcup_{i=1}^C D_i$. Then the set~$F=\{a\in A\,|\,m\cdot a\in F'\}$ is finite and~$a_n\in F+\bigcup_{i=1}^C D_i$.

It follows that we may substitute~$a_n$ with the~$a'_n$, and assume~$\beta_1=\ldots=\beta_r=0$.

We argue by contradiction. In particular we may pass to a subsequence of $(a_n)_n$ for which, for every $i=1, \dots, C$, we have
\begin{displaymath}
 \ord (a_n + D_i/D_i)\to + \infty,
\end{displaymath}
as $n \to + \infty$. The strategy is to find some $1\leq i\leq C$ and a non-zero abelian subvariety $A'\subset A/D_i$ such that $A'_{\mathfrak{p}}\subseteq \mathfrak{B}_i/D_{i,\mathfrak{p}}$. The preimage of $A'$ along the quotient map $A\to A/D_i$ would then contradict the maximality of $D_i$.

For every $n$ there exists an $1 \leq i \leq C$ such that
\begin{displaymath}
\# \{ \pi_{\mathfrak{p}}(\Gal (\overline{K}/K) \cdot a_n) \cap \mathfrak{B}_i \} \geq \frac{1}{C} \# \{ \Gal (\overline{K}/K) \cdot a_n\}.
\end{displaymath}
Extracting a subsequence, we may assume that $i$ is fixed. For simplicity set $D=D_i$, $\mathfrak{B}=\mathfrak{B}_i$ and 
\begin{displaymath}
d_n= \ord (a_n), \text{  and  } d_n' =\ord (a_n + D/D).
\end{displaymath}
By assumption~$d_n '$ goes to infinity. Possibly extracting a subsequence, at least one of the following is true: 
\begin{enumerate}
\item There exist a rational prime number $\ell (\neq p)$ and a sequence $(k_n) \subset \Z_{\geq 0}$ going to infinity, as $n\to + \infty$, such that
\begin{displaymath}
\forall n, \ell ^{k_n} \text{  divides  } d_n'.
\end{displaymath}
\item There exists an infinite sequence of distinct prime numbers $\ell _n (\neq p)$  such that 
\begin{displaymath}
\forall n, \ell_n \text{  divides  } d_n'.
\end{displaymath}
\end{enumerate}
We argue each case separately (the strategy is the same but the techniques are different).
\begin{proof}[Proof in Case (1)]
We set
\begin{displaymath}
a_n':= [d_n' / \ell^{k_n}] a_n + D/D,
\end{displaymath}
which is an element of order $\ell^{k_n}$ of $A/D$. We say that a sequence $ b_n \in A/D[\ell^{k_n}]$, of elements of order $\ell^{k_n}$, converges to some $b_\infty \in T_\ell (A/D)$, the $\Z_\ell$-Tate module of $A/D$, if there exist $\epsilon_n$ converging to $0 \in T_\ell (A/D)$ such that $b_n$ is the reduction modulo $\ell^{k_n}$ of $b_\infty + \epsilon_n$. Thanks to the compactness of $T_\ell (A/D)$, we may assume that $a'_n$ converges to some non-zero $a'_\infty \in T_\ell (A/D)$ (indeed we can pick some lifts $\tilde{a}'_n\in T_\ell (A/D)$ of $a_n'$ and write $\tilde{a}_n'=\tilde{a}'_\infty + \epsilon_n$, with $\epsilon_n \to 0$ in $T_\ell (A/D)$).

Let $\mathcal{G}_\ell$ be the image of the natural continuous representation
\begin{displaymath}
\Gal (\overline{K}/K)\to \Gl (V_\ell (A/D)).
\end{displaymath}
For all $n$, we have\footnote{By abuse of notation we still denote by $\pi_{\mathfrak{p}}$ the reduction mod $\mathfrak{p}$ map $A/D (\overline{K})\to \mathcal{A}_{\mathfrak{p}}/D_{\mathfrak{p}}(\overline{\kappa(\mathfrak{p})}) $, induced by the map $\pi_{\mathfrak{p}} : A(\overline{K})\to \mathcal{A}_{\mathfrak{p}}(\kappa(\mathfrak{p}))$.}
\begin{equation}\label{eqaution1c}
\# \{( \pi_{\mathfrak{p}}(\Gal (\overline{K}/K) \cdot a_n')) \cap \mathfrak{B}/D_\mathfrak{p} \}\geq \frac{1}{C} \# \{ \Gal (\overline{K}/K) \cdot a'_n\}.
\end{equation}

We define the set
\begin{displaymath}
\Sigma_n :=\{\sigma \in \mathcal{G}_\ell : \pi_{\mathfrak{p}}(\sigma \cdot  a'_n) \in \mathfrak{B} /D_{\mathfrak{p}}  \}.
\end{displaymath}
For any sequence $(\sigma _n)_n \in \mathcal{G}_\ell $ such that $\sigma_n \in \Sigma_n$ for every $n$, for any limit of any subsequence (for the $\ell$-adic topology) $\sigma_\infty$ we have, along such subsequence
\begin{displaymath}
\sigma_n \cdot  a'_n \to \sigma_\infty \cdot  a'_\infty.
\end{displaymath} 
As $ \pi_{\mathfrak{p}}(\sigma \cdot  a'_n)$ belongs to $\mathfrak{B} /D_{\mathfrak{p}}$, we see that
\begin{displaymath}
 \sigma_\infty \cdot  a'_\infty \in V_{\ell}(\mathfrak{B} /D_{ \mathfrak{p}}) \subseteq  V_{\ell}(A /D),
\end{displaymath}
where we identify $V_{\ell}(A /D)$ with $V_{\ell}(\mathcal{A}_\mathfrak{p} /D_\mathfrak{p})$ (since $p$ different from $\ell$). We define $\Sigma_\infty$ to be the set of the $\sigma_\infty \in \mathcal{G}_\ell$ obtained as above.

Let $\mu$ be the Haar probability measure on $\mathcal{G}_\ell$. From the equation \eqref{eqaution1c}, we see that $\mu (\Sigma_n)$ has Haar measure at least $(1/C) $, for every $n$. Therefore $\Sigma_\infty$ has measure at least $1/C$. Let $H_\ell/\Q_\ell$ be the Zariski closure of $\mathcal{G}_\ell$ in $\Gl (V_\ell (A/D)) $ and $H_\ell^0$ its neutral component. Recall that, from the Bogomolov-Serre theorem \cite{bog}, $\mathcal{G}_\ell$ has finite index in $H_\ell(\Z_\ell)$. The Zariski closure of $\Sigma_\infty$ in $\operatorname{GL}(V_\ell(A/D))$ contains one irreducible component $H_\ell^0 \cdot \sigma_0$, for some $\sigma_0 \in \Sigma_\infty$. Let $U$ be the subvector space of $V_{\ell}(A /D)$ given by $V_{\ell}(\mathfrak{B} /D_{ \mathfrak{p}})$, and $W$ be the $\Q_\ell$-subvector space generated by $H_\ell^0(\Q_\ell) \cdot \sigma_0 \cdot  a'_\infty$. 

By Zariski density, we have that $H_\ell^0 (\Q_\ell) \cdot \sigma_0 \cdot a'_\infty \subset U$. Notice that, by construction, $W$ is $H_\ell^0$ invariant. Let $\mathcal{G}_\ell^0= \mathcal{G}_\ell \cap H_\ell^0$. It is a finite index subgroup of $\mathcal{G}_\ell$ which corresponds to a finite extension of fields $K_\ell /K$. We may apply Proposition \ref{lemmatolift}, with $E=\Q_\ell$-coefficients to $W$ and $U$, to obtain a $ K_\ell$-abelian subvariety $A'' \subseteq (A/D)_{ K_\ell} $ such that 
\begin{equation}\label{r}
V_{\ell}(A'')\subseteq V_\ell (\mathfrak{B} /D_{ \mathfrak{p}}),
\end{equation}
and 
\begin{displaymath}
\Gal(\overline{K}/K_\ell) \cdot \sigma \cdot a_\infty' \subset W \subseteq V_{\ell}(A'').
\end{displaymath}
By the reduction of Section \ref{reduction}, we have that the same holds true for some abelian subvariety $A' /K \subseteq A/D$.
By definition of $D$, thanks to \eqref{r}, we see that $A'=0$, but since $0 \neq a'_\infty \in V_\ell( A')$, we must also have that $A'\neq 0$. This is the contradiction we were aiming for.
\end{proof}

\begin{proof}[Proof in Case (2)]
In this case, we set
\begin{displaymath}
a_n':= [d_n' / \ell_n] a_n + D/D,
\end{displaymath}
which is an element of order $\ell_n$ of $A/D$. Consider the ring $R:= \prod_{n\geq 0} \Ff_{\ell_n}.$
The sequence $(a_n')_n$ is naturally an element of $T_{\widehat{\Z}}(A/D)\otimes R$, where $T_{\widehat{\Z}}(A/D)$ is the adelic Tate module of $A/D$. By assumption, we know that the order of $a_n'$ is going to infinity, as $n\to + \infty$, and in particular $(a_n')_n$ is not a torsion element. The image of $(a_n')_n$ in 
\begin{displaymath}
T_{\widehat{\Z}}(A/D)\otimes R \otimes \Q \cong (R\otimes \Q)^{2 \dim (A/D)}
\end{displaymath}
is non-zero. There exists a field $E$ and a ring map $R\otimes \Q\to E$ such that the image $a'_\infty \in T_{\widehat{\Z}}(A/D)\otimes E $ is non-zero.

Notice that, for every $n$,
\begin{displaymath}
\# \{( \pi_{\mathfrak{p}}(\Gal (\overline{K}/K) \cdot a'_n) \cap \mathfrak{B}/D_{ \mathfrak{p}})\}\geq \frac{1}{C} \# \{ \Gal (\overline{K}/K) \cdot a_n'\}.
\end{displaymath}
Let $G_{\ell,n}$ be the image of $\Gal (\overline{K}/K)$ in $\Gl (A/D[\ell_n ])$. We can apply the combinatorial result of Section \ref{sectiononcomblemma} (with $V= \mathfrak{B}/D_{ \mathfrak{p}} [\ell_n]$) to guarantee the existence a subgroup $H_{\ell, n}$ of $G_{\ell,n}$, and a $\sigma_n \in \Gal (\overline{K}/K)$, such that
\begin{displaymath}
[G_{\ell, n}:H_{\ell,n}]\leq \psi(C) = 3 \cdot C^{4^ {2\dim A/D}},
\end{displaymath}
and
\begin{displaymath}
\pi_{\mathfrak{p}} (H_{\ell, n} \cdot \sigma_n \cdot a_n') \subset \mathfrak{B}/D_{ \mathfrak{p}}.
\end{displaymath}
Let $K_{H_{\ell,n}}/K$ be the finite extension of fields corresponding to $H_{\ell, n}$, and notice that $[K_H:\Q] \leq \psi(C)[K:\Q]$. Let $M$ be the constant appearing in \eqref{mwthm} from Theorem \ref{tatethm}, for an abelian variety $A/D$ over $K_{H_{\ell,n}}$, and notice that $M$ is bounded independently of $n$. Upon extraction, we may assume that $\ell_n$ is strictly bigger than $M$ for every $n$. Let $W_n$ be the $\Ff_{\ell_n}$-subvector space of $A/D[\ell_n]$ generated by $H_{\ell, n}\cdot \sigma_n \cdot a'_n$. Notice that $W_n$ is invariant under the action of $\Gal (\overline{K}/K_{H_{\ell, n}})$, and so by the semisimplicity of \eqref{mwthm}, there exists an element  $f\in \End (A)\otimes \Ff_{\ell_n}$, such that 
\begin{displaymath}
W_n = f (A[\ell_n]).
\end{displaymath}
Since $K$  is such that all endomorphism of $A_{\overline{K}}$ are defined over $K$, we see that $W_n$ is actually $\Gal (\overline{K}/K)$-invariant. So we have, for every $n$,
\begin{displaymath}
\pi_{\mathfrak{p}}( \Gal (\overline{K}/K) \cdot \sigma_n \cdot  a_n')  \subset \mathfrak{B} /D_{ \mathfrak{p}}.
\end{displaymath}
In particular we see that
\begin{displaymath}
\Gal (\overline{K}/K) \cdot \sigma_\infty \cdot  a'_\infty  \subset  V_{E}(\mathfrak{B}/D_{ \mathfrak{p}}) \subset V_{E}(A /D).
\end{displaymath}
We now conclude as in the previous case. Let indeed $W$ be the $E$-subvector space generated by $\Gal (\overline{K}/K) \cdot a'_\infty$. We apply Proposition \ref{lemmatoliftadelic}, with $E$-coefficients to $W$ and $u\in \End ( \mathcal{A}_\mathfrak{p}/ D_{ \mathfrak{p}})\otimes \Q$ an idempotent corresponding to $\mathfrak{B} / D_{ \mathfrak{p}} \subset \mathcal{A}_\mathfrak{p}/ D_{ \mathfrak{p}}$. So we get an abelian subvariety $A'\subseteq A/D$ such that 
\begin{equation}\label{eqqqqqq}
V_{E}(A')\subseteq V_{E} (\mathfrak{B} / D_{\mathfrak{p}}),
\end{equation}
and 
\begin{displaymath}
\Gal(\overline{K}/K) \cdot a_\infty ' \subset W \subseteq V_{E}(A').
\end{displaymath}
By definition of $D$, thanks to \eqref{eqqqqqq}, we see that $A'=0$, but since $0 \neq a' _\infty \in V_E(A')$, we must also have that $A'\neq 0$. This is the contradiction we were aiming for.
\end{proof}
The proof of Proposition \ref{prop222} is complete.
\end{proof}

%\subsubsection{Proof of Theorem \ref{proplift}}\label{prooof}
%Let $E_n$ be a sequence of special zero cycles in $\mathcal{A}$, each one associated to some torsion point $a_n$ of order coprime to $p$. We may assume that the $a_n$'s are all distinct, and of order $>1$. In particular $\ord (a_n)$ is unbounded, and we will only consider the positive dimensional components of the Zariski closure of the $E_n$'s. We can apply the weaker Theorem \ref{mainthm} to the Zariski closure of the $0$-cycles $E_n$. It is of the form
%\begin{displaymath}
%\bigcup _{i=1}^C \alpha_i+\mathfrak{B}_i,
%\end{displaymath}
%for some positive integer $C$ which denotes the number of components, where each $\mathfrak{B}_i$ is an abelian subvariety of $\mathcal{A}_\mathfrak{p}$, and $\alpha_i$ torsion points. Let $A_i$ be the biggest special subvariety of $A$ whose reduction mod $\mathfrak{p} $ is in $\mathfrak{B}_i$. Applying Proposition \ref{prop222} we have that each $a_n$ lies in the union of the $A_i$, up to a finite set $F$. We have
%\begin{displaymath}
%E_n=\pi_\mathfrak{p}(\Gal (\overline{K}/K)\cdot a_n)\subset F + \bigcup _{i=1}^C A_{i,\mathfrak{p}} \subset \bigcup _{i=1}^C \alpha_i+\mathfrak{B}_i.
%\end{displaymath}
%Since $\bigcup _{i=1}^C \alpha_i+\mathfrak{B}_i$ is the Zariski closure of the $E_n$, we conclude and we observe that
%\begin{displaymath}
%F + \bigcup _{i=1}^C A_{i,\mathfrak{p}}= \bigcup _{i=1}^C \alpha_i+\mathfrak{B}_i.
%\end{displaymath}

We can now prove~Theorem \ref{proplift}.

\subsubsection{Reduction of Theorem \ref{proplift} to Corollary \ref{cor 1.2.5}}\label{prooof} %[IN WRITING!!]
We first reduce Theorem \ref{proplift} to Corollary \ref{cor 1.2.5}, using the classical Manin-Mumford theorem, Theorem \ref{mm}, which is also consequence of Theorem\ref{mainthm}, as noticed in Remark \ref{rmkimplication}.

Let~$(a_n)_{n\geq0}$ be a sequence in~$A(\overline{K})_{\operatorname{tors}}$, where each~$a_n$ is of order coprime to~$p$. Thanks to the classical Manin-Mumford theorem we may write,
in the abelian~$\overline{K}$-variety~$A$,
\begin{equation}\label{using mm}
\overline{\bigcup_{n\in\Z_{\geq0}} \Gal(\overline{K}/K)\cdot a_n}^{\Zar}=\bigcup_{i=1}^r b_i+A_i
\end{equation}
for some~$b_1,\ldots,b_r\in A(\overline{K})_{tors}$ and abelian subvarieties~$A_1,\ldots,A_r\subset A$.

%We will prove the following precise form of Theorem \ref{proplift}.
%\begin{proposition}
% Let~$E_n$ and~$a_n$ be as in Theorem \ref{proplift} and let~$a_i$, $B_i$ as in~\eqref{using mm}.
%We have
%\[
%\overline{\bigcup E_n}=\bigcup_{i=1}^r \pi_{\mathfrak{p}}(b_i)+(A_i)_{\mathfrak{p}}.
%\]
%\end{proposition}

There is a finite extension~$K'/K$ such that the~$b_i$~are in~$A(K')$ and the~$A_i$ are defined over~$K'$. We can write~$\Gal(\overline{K}/K)= \Gal(\overline{K}/K')\sigma_1\cup\ldots\cup \Gal(\overline{K}/K')\sigma_{[K:K']}$, and observe that Theorem \ref{proplift} for the sequence~$(a_n)_{n\in\Z_{\geq0}}$ over the field~$K$ is equivalent to Theorem \ref{proplift} for the sequence~$(\sigma_i(a_n))_{n\in\Z_{\geq0},1\leq i\leq[K':K]}$ over the field~$K'$. Therefore we may and do assume~$K'=K$. In particular, for~$a\in b_i+A_i(\overline{K})$ we have~$  \Gal(\overline{K}/K)\cdot a\subseteq b_i+A_i(\overline{K})$. Thus, for~$1\leq i\leq r$,
\begin{equation}\label{using K=K'}
\left(\bigcup_{n\in\Z_{\geq0}} \Gal(\overline{K}/K)\cdot a_n\right)\cap  (b_i+A_i)=\bigcup_{a_n\in b_i+A_i}\Gal(\overline{K}/K)\cdot a_n.
\end{equation}

We also assume that, in~\eqref{using mm}, each~$b_i+A_i$ is an irreducible component of the variety. It follows that, for~$1\leq i\leq r$,
\[
b_i+A_i=\overline{\bigcup_{a_n\in b_i+A_i} \Gal(\overline{K}/K)\cdot a_n}^{\Zar}.
\]

It is enough to prove Theorem \ref{proplift} for each of the subsequences~$(a_m)_{m\in\{n\in\Z_{\geq0}|a_n\in b_i+A_i\}}$. 
We may thus assume that$r$, the number of irreducible components, is $=1$. Substituting~$a_n$ by~$a_n-b_1$ , we may also assume~$b_1=0$.

Replacing~$A$ by~$A_1$, it is enough to prove Corollary \ref{cor 1.2.5}.

\subsubsection{Proof of Corollary \ref{cor 1.2.5}}\label{prooof-cor}
By passing to a subsequence, we may assume that the sequence $(a_n)_{n\geq0}$ is
is \emph{generic} in~$A$. In particular, for every strict abelian subvariety~$D\subsetneq A$,
the sequence~$(a_n+D)_{n\in\Z_{\geq0}}$ of~$A/D$ satisfies 
\begin{equation}\label{etre strict}
\ord(a_n+D)\to +\infty.
\end{equation}
 
By Theorem~\ref{mainthm} we have
\[
\overline{\bigcup_{n\in\Z_{\geq0}} E_n}^{\Zar}= \bigcup _{i=1}^C \beta_i+\mathfrak{B}_i \subset \mathcal{A}_{\mathfrak{p}}.
\]

In fact, thanks to Proposition \ref{prop222}, there is a finite set $F$ such that
\begin{equation}\label{from prop222}
a_n \in F +\bigcup_{i=1}^C D_i.
\end{equation}

If every~$D_i\subseteq A$ is a strict subvariety, then~\eqref{etre strict} implies
\[
\forall 1\leq i\leq C, \ord(a_n+D_i)\to+\infty,
\]
contradicting~\eqref{from prop222}. This implies that the special 0-cycles $(E_n)_{n\geq 0}$ are Zariski dense in $\mathcal{A}_\mathfrak{p}$ We have proved Corollary ~\ref{cor 1.2.5}, and this finishes the proof of Theorem \ref{proplift}.

\section{Purely \texorpdfstring{$p$}{p}-adic phenomena, Theorem \ref{pureppart}}\label{padicsection}
In this final section we prove Theorem \ref{pureppart}, which, for the convenience of the reader, is recalled below. %The most general result which combines Theorem \ref{mainthmlift} and Theorem \ref{pureppart} is also proven here as Corollary \ref{finalcor}.
\begin{thm}[Theorem \ref{pureppart}]\label{pureppartrecall}
Fix a prime $\mathfrak{p}\subset \Oo_{K,S}$ of residue characteristic $p$. Let $\mathcal{A}$ be an abelian scheme over $\Oo_{K,S}$, $(a_n)_n\subset A_{\operatorname{tors}} (\overline{K})$  (possibly of order divisible by $p$), and write $E_n=\pi_{\mathfrak{p}}(\Gal(\overline{K}/K) \cdot a_n)$, for a map $
\pi_{\mathfrak{p}}: A(\overline{K})\to \mathcal{A}_\mathfrak{p}(\overline{\kappa(\mathfrak{p})}).$
Then for every component $\alpha + \mathfrak{B}$ of the Zariski closure $\mathcal{V}$ of 
\begin{displaymath}
\bigcup_{n\geq 0} E_n\subset \mathcal{A}_\mathfrak{p}\subset \mathcal{A}
\end{displaymath}
there exists $B \subseteq A$ an abelian subvariety such that $ \mathfrak{B} \subset B_\mathfrak{p}$ and the quotient $B_\mathfrak{p}/  \mathfrak{B}$ is of $p$-rank zero.
\end{thm}

The key input needed to adapt the strategy of the proof of Theorem \ref{proplift} to the above, is the following $p$-adic version of the arguments of Section \ref{keypropo}. Compared to the $\ell$-adic argument, the crux is to have some control on the reduction map from $T_p(A)$ to $T_p(\mathcal{A}_\mathfrak{p})$ and to find a suitable faithful representation of the endomorphism algebra of $A$ and $\mathcal{A}_\mathfrak{p}$. All this is summarised in the statement of Proposition \ref{thingy statement}, which we first use as a black box, and then explain its proof (in the appendix).
\begin{thm}\label{thm502}
Let $W$ be a $\Gal(\overline{K}/K)$-invariant subspace of $V_p(A)$ and denote by $W_\mathfrak{p}$ its image in $V_p(\mathcal{A}_\mathfrak{p})$. If there exists $u \in \End(\mathcal{A}_\mathfrak{p})\otimes \Q$ such that
\begin{displaymath}
W_\mathfrak{p}\subseteq u\cdot V_p(\mathcal{A}_\mathfrak{p}),
\end{displaymath}
then there is an idempotent $v\in \End(A)\otimes \Q$ such that
\begin{displaymath}
W_\mathfrak{p} \subseteq v\cdot V_p(\mathcal{A}_\mathfrak{p})\subseteq u\cdot V_p(\mathcal{A}_\mathfrak{p}).
\end{displaymath}

Assume moreover that $ u\cdot V_p(\mathcal{A}_\mathfrak{p})=V_p(\mathfrak{B})$ where $\mathfrak{B} \subset \mathcal{A}_\mathfrak{p}$ is the smallest abelian subvariety such that $W_\mathfrak{p} \subset V_p(\mathfrak{B})$.
Then $v\cdot V_p(A)= V_p(B)$ and $v\cdot V_p(\mathcal{A}_\mathfrak{p})= V_p(B_{\mathfrak{p}})$ with~$B=vA$, as an abelian subvariety. Furthermore %for an abelian subvariety $B \subseteq A$ such that 
$ \mathfrak{B} \subset B_\mathfrak{p}$, and the quotient $B_\mathfrak{p}/  \mathfrak{B}$ is of $p$-rank zero (i.e. $B_\mathfrak{p} [p^\infty]= \mathfrak{B}[p^\infty]$).
\end{thm}

\subsection{Preliminaries}
\subsubsection{Statement of a key proposition}
To prove Theorem \ref{thm502} we need the following result about abelian varieties $A$ over local fields $L:=K_{\mathfrak{p}}$ with good reduction over the residue field $k:=\kappa(\mathfrak{p})$. Set $V:=V_p(A_L)$.
\begin{prop}\label{thingy statement}
There exist a field extension $E$ of $L$, and a polynomial $Q(X)$ with coefficients in $\QQbar \cap \Q_p$ such that:
\begin{enumerate}
\item The left action $\End(A_L)$ on $V\tens_{\Q_p} E$ extends to an $E$-linear left action
\begin{displaymath}
\rho:\End(A_k)\to \End_E(V\tens_{\Q_p} E);
\end{displaymath}
\item The kernel of the reduction map, which we call \emph{the ramified part of V} and denote by
\begin{displaymath}
V^r:=\ker  \left(V_p(A_L) \to V_p(A_k)\right),
\end{displaymath}
and the Frobenius endomorphism $\FrobAk\in\End(A_k)$ satisfy the identity of $E$-linear subspaces
\begin{displaymath}
\pi(V\tens E)=V^r\tens E,
\end{displaymath}
where $\pi=Q(\rho(\FrobAk))$ is seen as an element of $\End_E(V\tens_{\Q_p} E)$. Moreover $\pi$ is a central idempotent;
\item Consider $V\otimes E $ modulo $V^r\otimes E$. Via $\rho$, it is an $\End(A_k) \otimes E$-module, and it can naturally be identified with the $\End(A_k) \otimes E$-module $V_p(A_k)\otimes E$ (where the action on the latter is the natural one coming from the naïve $p$-adic Tate module of $A_k$, as an $\End(A_k)$-module).
\end{enumerate}
\end{prop}
The proof appears in Section \ref{finalsectionp}, after all required objects are introduced and described (for example $E$ is one of Fontaine's famous period rings).

\subsubsection{Variant of Lemma \ref{keylemma} (with an extra central idempotent)}
%The variant of Lemma \ref{keylemma} needed for the previous argument is the following.
\begin{prop}\label{keylemma22}Let $M, N$ be semisimple finite dimensional $\Q$-algebras, with $M$ a subalgebra of $N$. Let $E/\Q$ be a field extension, $\pi \in N \otimes E$ a central idempotent, and
\begin{displaymath}
\left(N\otimes_\Q E \right)\hookrightarrow \End_E(V)
\end{displaymath}
be a faithful representation in $V$ a finite dimensional $E$ linear vector space. Let $u$ be an idempotent of $N$, and $w$ be an idempotent of $M\otimes_\Q E$ such that
\begin{displaymath}
w \cdot V+ \pi \cdot V \subseteq  u\cdot V + \pi \cdot V.
\end{displaymath}
Then there exists an idempotent $v\in M$, such that
\begin{displaymath}
w\cdot V + \pi \cdot V \subseteq v\cdot V + \pi \cdot V \subseteq u\cdot V+ \pi \cdot V.
\end{displaymath}
\end{prop}
\begin{rmk}\label{r1}
Let $B$ be a semisimple algebra, and $\pi$ be a \emph{central} idempotent. Then we can write, as $\Q$-algebras 
\begin{displaymath}
B=\pi B \oplus (1_B-\pi)B.
\end{displaymath}
Let $I$ be a left (resp. right) ideal of $B$ containing $\pi$, we may write
\begin{displaymath}
I= \pi B \oplus I',
\end{displaymath}
for $I':= (1_B-\pi)I$. So we have a correspondence between left (resp. right) ideals of $B$ containing $\pi$, and left (resp. right) ideals of $(1_B-\pi)B$.\end{rmk}
For the proof of Proposition \ref{keylemma22}, we need two lemmas (the first is a simple fact).
\begin{lemma}\label{lem1}
Let $M\subset B$ be a semisimple algebras, $\pi$ a central idempotent such that $B=M[\pi]$. Then $M$ surjects onto $B/(\pi)$, and every idempotent of $B/(\pi)$, comes from an idempotent of $M$.
\end{lemma}

\begin{lemma}\label{lem2}
Let $M, B , \pi$ be as in Lemma \ref{lem1}, and $V$ be a $\Q$-vector space with a faithful action of $B$. Let $u$ be an idempotent of $B$, and consider the ideal of $B$ 
\begin{displaymath}
I:=uB+\pi B\subset B.
\end{displaymath}
An element $b\in B$ belongs to $I$ if and only if 
\begin{displaymath}
b \cdot V +  \pi \cdot V \subseteq u \cdot V + \pi \cdot V.
\end{displaymath}
\end{lemma}
\begin{proof}
We need only to prove $\Leftarrow$, the other implication is evident. Let $b\in B$ be such that $b \cdot V + \pi \cdot V \subseteq u \cdot V + \pi \cdot V$. We write 
\begin{displaymath}
b = \pi b + (1-\pi)b.
\end{displaymath}
We also have an induced decomposition on $V$:
\begin{displaymath}
V= \ker (\pi \cdot ) \oplus \im (\pi \cdot).
\end{displaymath}
As $u$ and $b $ commute with $\pi$, they are compatible with the above direct sum decomposition, that is $u$ acts on $\ker (\pi \cdot )$ as $(1-\pi)u$ and on $\im (\pi \cdot)$ as $\pi u$, and similarly for $b$. We have 
\begin{displaymath}
\im ((1-\pi) b ) \oplus \im (\pi)= \im (b) + \im (\pi)  \subseteq \im (u) + \im (\pi)= \im ((1-\pi)u)\oplus \im (\pi).
\end{displaymath}
Hence $\im ((1-\pi) b ) \subseteq \im ((1-\pi) u)$. But $(1-\pi) u$ is unipotent, by Lemma \ref{keylemma}, $(1-\pi)b$ belongs to $(1-\pi)u B$, say 
\begin{displaymath}
(1-\pi)b = (1-\pi)u x,
\end{displaymath}
for some $x \in B$. Finally $b=\pi b +(1-\pi)b=\pi b + (1-\pi)ux= \pi b + u(1-\pi)x\in I$.
\end{proof}

\begin{proof}[Proof of Proposition \ref{keylemma22}] 
We first consider
\begin{displaymath}
X_\pi:= (uN[\pi]+\pi N[\pi]) \cap (M [\pi]),
\end{displaymath}
where $u\in N$ acts diagonally on $N[\pi]$. It is a right ideal of $M[\pi]$, and as such it is generated by an idempotent $m=m^2$ of $M[\pi]$ (as recalled before the proof of Lemma \ref{keylemma}):
\begin{displaymath}
X_\pi = m M[\pi].
\end{displaymath}
Notice that $X_\pi$ contains $(\pi)$, and so, by Lemma \ref{lem1} (applied to $B=M[\pi]$), we can write $X_\pi =v M[\pi]+ \pi M [\pi] $, where $v$ is an idempotent of $M$ such that
\begin{displaymath}
v \mod (\pi) = m \mod (\pi),
\end{displaymath}
that is $v+ (\pi)= m + (\pi)$. Hence $m M[\pi]= m M[\pi]  + \pi m M[\pi]  = v M[\pi]+ \pi M[\pi]$.

Notice that 
\begin{displaymath}
X_\pi \otimes E= ((uN[\pi]+\pi N[\pi])\otimes E) \cap (M [\pi]\otimes E).
\end{displaymath}
We may apply Lemma \ref{lem2} to $b=v \in I=uN[\pi]\otimes E + \pi N[\pi]\otimes E$, to get 
\begin{displaymath}
v\cdot V + \pi \cdot V \subseteq u\cdot V+ \pi \cdot V.
\end{displaymath}
Again applying Lemma \ref{lem2} $w \in uN[\pi]\otimes E + \pi N[\pi]\otimes E$, since we know that
\begin{displaymath}
w \cdot V+ \pi \cdot V \subseteq  u\cdot V + \pi \cdot V.
\end{displaymath}
By construction $w\in M [\pi]\otimes E$, so $w \in X_\pi \otimes E$, and we can write
\begin{displaymath}
w\cdot V + \pi \cdot V \subseteq v\cdot V + \pi \cdot V.
\end{displaymath}
\end{proof}

\subsection{Proof of Theorem \ref{thm502}}\label{sectionp}
 The place $\mathfrak{p}$ is a place of good reduction for $A$ and have
\begin{displaymath}
\End_{K}(A_{K})= \End_{K_{\mathfrak{p}}}(A_{K_{\mathfrak{p}}}),
\end{displaymath} 
since the endomorphism ring of an abelian variety does not change under extensions of algebraically closed field of characteristic zero\footnote{This follows from the fact that the functor $\End_K(A)$ is representable by an \'{e}tale commutative $K$-group scheme.} (recall also the reductions from Section \ref{reduction} that guarantee that $\End_{K}(A_{K})= \End_{\overline{K}}(A_{\overline{K}})$).
As a $\Q_p$-vector spaces of dimension $2g$, we may also naturally identify $V_p(A_{K})$ with $V_p(A_{K_{\mathfrak{p}}})$, via the map induced by $\overline{K} \to\overline{K_{\mathfrak{p}}}$. We now use some facts that will be explained in details in Appendix \ref{preliminaries} (where we let $L$ to be $K_{\mathfrak{p}}$). Set $V:= V_p(A_{K})$. Let $E/\Q$ be the field extension of Proposition \ref{thingy statement}, and consider $V\otimes E$ with its natural action of $\End(A_{K})\otimes \Qp$ and the $\rho$-action of $\End(\mathcal{A}_\mathfrak{p})\otimes \Qp$ coming from (1) of Proposition \ref{thingy statement}.

We can consider the element 
\begin{equation}\label{defofpi}
\pi\in \left( \End(\mathcal{A}_\mathfrak{p})\otimes \Qbar \right) \cap \left( \End(\mathcal{A}_\mathfrak{p})\otimes \Q_p\right),
\end{equation}
defined in (2) of Proposition \ref{thingy statement}. By construction, $\pi$ is a central element (it is a $(\Q_p\cap \Qbar)$-combination of powers of the Frobenius endomorphism $F_{A,\mathfrak{p}}\in \End(\mathcal{A}_\mathfrak{p})$). Thanks to \eqref{tateconj1} (the Tate conjecture with $\Q_p$-coefficients), there is some idempotent $w\in \End(A_{K})\otimes \Qp$ such that $W \otimes E =w \cdot V\otimes E$. The key point of Proposition \ref{thingy statement} is that
\begin{equation}\label{eqqqqqqq}
\rho(\pi)(V\otimes E)=V^r\otimes E
\end{equation}
where $V^r= \ker \left( V \to V_p(A_{\kappa(\mathfrak{p})}) \right)$. By assumption, $W$ modulo $V^r$ is in the naïve Tate module of $\mathfrak{B}$, which is $u\cdot V_p(A_{\kappa(\mathfrak{p})})$. Thanks to (3) of Proposition \ref{thingy statement}, we may identify the $W\otimes E$, modulo $V^r\otimes E$, with $\rho (u)\cdot V \otimes E$.

Thanks to the definition of $\pi$ from \eqref{defofpi}, and \eqref{eqqqqqqq}, we have an inclusion, modulo $V^r$:
\begin{equation}
\im (w) \otimes E \subseteq  \im (\rho (u)).
\end{equation}
Proposition \ref{keylemma22} guarantees the existence of an idempotent $v\in \End (A_{K})\otimes\Q$ such that 
\begin{displaymath}
w\cdot  V\otimes E+ V^r\otimes E \subseteq v\cdot V \otimes E + V^r\otimes E\subseteq \rho (u) \cdot V\otimes E + V^r\otimes E.
\end{displaymath}
The image of $v$ gives, up to a quasi-isogeny, the desired abelian subvariety $B\subset A$, satisfying 
\begin{equation}
W_{\mathfrak{p}} \subset T_p(B_\mathfrak{p})\subset T_p(\mathfrak{B}).
\end{equation}

For the final part of the statement, assume that $ u\cdot V_p(\mathcal{A}_\mathfrak{p})=V_p(\mathfrak{B})$ where $\mathfrak{B} \subset \mathcal{A}_\mathfrak{p}$ is the smallest abelian subvariety such that $W_\mathfrak{p} \subset V_p(\mathfrak{B})$.
By the above construction of $B$ (and $B_\mathfrak{p}$) we must have $T_p(B_{\mathfrak{p}})=T_p(\mathfrak{B})$ and $\mathfrak{B} \subset B_\mathfrak{p} $. Therefore $B_\mathfrak{p} [p^\infty]= \mathfrak{B}[p^\infty]$, and so the quotient abelian variety $B_\mathfrak{p}/ \mathfrak{B}$ is of $p$-rank zero, concluding the proof of Theorem \ref{thm502}.

\subsection{Proof of Theorem \ref{pureppart}}\label{proofofthefinalpthm}
Theorem \ref{pureppart} will be proven in Section \ref{finalproof}, after we describe a version of Proposition \ref{prop222} and we prove the the $p$-power order case of the theorem.

Recall that~$A$ is an abelian variety over a number field~$K$ such that~$\End_K(A)=\End_{\overline{K}}(A)$, and that~$\mathcal{A}_\mathfrak{p}$ denotes its reduction at a 
prime~$\mathfrak{p
}$ of good reduction. We denote by $\operatorname{red}$ the map of~$\Q_p$-Tate modules
\begin{displaymath}
\operatorname{red} : V_p (A)\to V_p(\mathcal{A}_\mathfrak{p}).
\end{displaymath}

Let us introduce the~$\Gal(\overline{K}/K)$-invariant~$\Q_p$-linear subspace
\[
L:=\bigcap_{\sigma\in \Gal(\overline{K}/K)} \sigma(\ker (\operatorname{red}))\subset V_p(A).
\]
In other words,~$L$ is the biggest~$\Gal(\overline{K}/K)$-invariant subspace contained in~$\ker (\operatorname{red})$.
Observe that the $\Gal(\overline{K}/K)$-invariant subpaces  of~$V_p(A)$ are also the images of the endomorphisms~$u\in \End_{K}(A)\tens\Q_p$ acting on~$V_p(A)$.
Since~$\End_K(A)=\End_{\overline{K}}(A)$, we have, for any finite extension~$E/K$, 
\begin{equation}\label{eq:L for E}
L=\bigcap_{\sigma\in \Gal(\overline{K}/E)} \sigma(\ker (\operatorname{red}))\subset  V_p(A).
\end{equation}

By semisimplicity of the Galois action (Theorem \ref{tatethm}), there exists a~$\Gal(\overline{K}/K)$-invariant~$\Q_p$-linear subspace~$L'\subset V_p(A)$
such that
\[
V_p(A)=L\oplus L'.
\]

Using the identification
\[
A[p^\infty]=V_p(A)/T_p(A)
\]
we define the subgroups
\begin{equation}\label{def:lambda prime}
\Lambda=L/T_p(A)\text{ and }\Lambda'=L'/T_p(A).
\end{equation}
We have~$\Lambda\simeq (\Q_p/\Z_p)^{\dim(L)}$ and~$\Lambda'\simeq (\Q_p/\Z_p)^{\dim(L')}$.

Since~$V_p(A)=L+ L'$ we have~$\#\Lambda\cap \Lambda'<+\infty$ and
\[
A[p^\infty]=\Lambda+\Lambda'.
\]

For~$a\in A[p^\infty]$ decomposed as~$a=a'+l$ with~$a'\in \Lambda'$ and~$l\in\Lambda$, 
we have
\[
\pi_\mathfrak{p}(\Gal(\overline{K}/K)\cdot a)=\pi_\mathfrak{p}(\Gal(\overline{K}/K)\cdot a').
\]

The following is a version of Proposition \ref{prop222} that works also for points of order divisible by $p$, the main difference is in the definition of the $D_i$. Here, by abuse of notation, we denote by $\operatorname{red}$ the map
\begin{displaymath}
\operatorname{red} : T_p (A)\to T_p(\mathcal{A}_\mathfrak{p}).
\end{displaymath}

\begin{prop}\label{prop2222}Let~$\Lambda'$ be as in~\eqref{def:lambda prime}, $(a_n)_n\subset \Lambda'$ and $(E_n)_n= \pi_{\mathfrak{p}}(\Gal(\overline{K}/K)\cdot a_n)$.
Let~$\mathfrak{B}_1,\ldots,\mathfrak{B}_C$  be abelian subvarieties of $\mathcal{A}_\mathfrak{p}$.
For $i=1, \dots , C$, let~$D_i\subset A$ be the biggest abelian subvariety such that~$T_p(D_{i,\mathfrak{p}})\subseteq T_p(\mathfrak{B}_i)$.

Assume that, for some~$\beta_1 ,\ldots,\beta_C$ in~$\mathcal{A}_{\mathfrak{p}}$, 
\[
\bigcup_n E_n\subseteq  \bigcup_{i=1}^C \beta_i +\mathfrak{B}_i.
\]
Then there exists a finite set $F\subset A$ such that for all $n$
\begin{displaymath}
a_n \in F +\bigcup_{i=1}^C D_i.
\end{displaymath}

\end{prop}
The proof of the above, follows the strategy of the proof of Proposition \ref{prop222}.
\begin{proof}
By multiplying by the product of the orders of the $\beta_i$, we may assume that every $\beta_i=0$. 

For every $n$ we can choose an $1 \leq i \leq C$ such that 
\begin{displaymath}
\# \{(\Gal (\overline{K}/K) \cdot a_n) \cap  \pi_{\mathfrak{p}}^{-1}(\mathfrak{B}_i) \} \geq \frac{1}{C} \# \{ \Gal (\overline{K}/K) \cdot a_n\}.
\end{displaymath}

We decompose the sequence~$(a_n)_n$ into finitely many subsequences according to the corresponding~$i$. 
We may argue with each subsequence independently. Without loss of generality we may assume that~$i$ 
does not depend on~$n$. We write $\mathfrak{B}= \mathfrak{B}_i$ and~$D=D_i$, and we have $\pi_{\mathfrak{p}}(a_n) \in \mathfrak{B}$, for every $n$.

We will prove that there exists a finite set~$F\subseteq A$ such that~$a_n\in F+D$ for all~$n$. The Proposition~\ref{prop2222}  will follow.

Heading for a contradiction, we assume that~$\ord (a_n + D) \to + \infty$, as $n\to + \infty$. As in (1) Proposition \ref{prop222}, we can construct a non-zero $a_\infty \in V_p (A/D)$, and a finite extension $E/K$ such that 
\begin{displaymath}
\Gal (\overline{K}/{E})\cdot a_\infty \subset \operatorname{red}^{-1}V_p(\mathfrak{B}+D/D).
\end{displaymath}
Since~$a_n\in \Lambda'$ for every~$n$, we have~$a_\infty\in L'$.
Let~$W\subset V_p(A)$ be the~$\Q_p$-linear subspace generated by~$\Gal (\overline{K}/{E})\cdot a_\infty$.

Since~$L'$ is~$\Gal (\overline{K}/{E})$-invariant, we have~$W\subset L'$. It implies~$W\cap L=\{0\}$.
By~\eqref{eq:L for E} we have,
\[
\forall w\in W\smallsetminus\{0\}, \pi_{\mathfrak{p}}(\Gal (\overline{K}/{E})\cdot w)\neq \{0\}.
\]
For~$w=a_\infty$, we get
\[
W_{\mathfrak{p}}:=\operatorname{red}(W)\neq \{0\}.
\]

Thanks to Theorem \ref{thm502}, there exits $A''\subset A/D$ such that 
\begin{displaymath}
W_{\mathfrak{p}}\subset \operatorname{red} V_p(A''_{\mathfrak{p}})\subset V_p(\mathfrak{B}+D_{\mathfrak{p}}/D_{\mathfrak{p}}).
\end{displaymath}

As~$W_{\mathfrak{p}}\neq\{0\}$, we must have~$A''\neq 0$.

Let~$D \subset  D' \subset  A$ be the abelian subvariety such that~$D'/D=A''$. Then~$V_p(D'_\mathfrak{p})\subset V_p(\mathfrak{B})$
and~$D<D'$. This contradicts the maximality property in the definition of~$D$. The proof of the proposition is completed.

%Since $a_\infty \neq 0$, and $D$ is the biggest abelian subvariety such that $\operatorname{red} T_p(D) \subset   T_p(\mathfrak{B})$. Therefore $W_{\mathfrak{p}}\neq 0$?
\end{proof}
\newcommand{\ol}[1]{\overline{#1}}

\subsubsection{The $p$-power order case of Theorem~\ref{pureppart}}
\begin{prop}\label{prop:particular}
Let~$a_n$ be a generic sequence in~$A$ such that~$a_n\in\Lambda'$ for every~$n$,
and let~$E_n=\pi_{\mathfrak{p}}(\Gal(\overline{\Q}/K)\cdot a_n)$.
Then
\[
\ol{\bigcup_{n\in\Z_{\geq0}} E_n}^{\Zar}=\overline{\mathcal{A}_{\mathfrak{p}}[p^{\infty}]}^{\Zar}.
\]
\end{prop}
\begin{proof}Since every $E_n$ lies in $A_{\mathfrak{p}}[p^{\infty}]$, we have $\ol{\bigcup_{n\in\Z_{\geq0}} E_n}^{\Zar}\subseteq \overline{\mathcal{A}_{\mathfrak{p}}[p^{\infty}]}^{\Zar}$. We now prove the other inclusion. Thanks to Theorem~\ref{mainthm} we may write, in~$\mathcal{A}_{\mathfrak{p}}$,
\[
\overline{\bigcup_{n\in\Z_{\geq0}} E_n}^{\Zar}= \bigcup _{i=1}^C \beta_i+\mathfrak{B}_i.
\]
For each~$i$, there exists~$n$ and~$\alpha\in E_n$ such that~$\alpha+\mathfrak{B}_i=\beta_i+\mathfrak{B}_i$.
Therefore, we may assume that~$\beta_i\in \mathcal{A}_{\mathfrak{p}}[p^{\infty}]$.

Thanks to Proposition \ref{prop2222}, there is a finite set $F$ such that
\begin{equation}\label{from prop2222}
a_n \in F +\bigcup_{i=1}^C D_i.
\end{equation}

As~$a_n$ is generic in~$A$, we have~$D_i=A$ for some~$i$. By definition of~$D_i$, we have
\[
\mathcal{A}_{\mathfrak{p}}[p^\infty]={\mathcal{D}_i}_{\mathfrak{p}}[p^\infty]\subseteq \mathfrak{B}_i.
\]
As~$\beta_i\in\mathcal{A}_{\mathfrak{p}}[p^{\infty}]$, we have~$\mathcal{A}_{\mathfrak{p}}[p^\infty]\subseteq \beta_i+\mathfrak{B}_i$. Thus
\[
\overline{\mathcal{A}_{\mathfrak{p}}[p^\infty]}^{\Zar}\subseteq \beta_i+\mathfrak{B}_i\subseteq \overline{\bigcup_{n\in\Z_{\geq0}} E_n}^{\Zar}.
\] 
We proved the conclusion of Proposition~\ref{prop:particular} by double inclusion.
\end{proof}

\subsubsection{Factoring the prime-to-$p$ order part and the power-of-$p$ order part}
\begin{lemma}\label{lem:2024}
Let~$A$ be an abelian variety over an algebraically closed field (in any characteritic). Let~$p$ 
be a prime and let~$\Sigma\subseteq A\times A$ be a subset such that
\[
\Sigma\subseteq \left(\sum_{\ell \neq p} A[\ell^{\infty}]\right)\times A[p^{\infty}].
\]
We denote by~$\pi_1,\pi_2:A\times A\to A$ the first and second projections.

If there exist $b_1,\ldots,b_k\in A_{\operatorname{tors}}\times A_{\operatorname{tors}}$ and abelian subvarieties
$B_1,\ldots,B_k\subset A\times A$ such that
\begin{equation}\label{lem2024eq1}
\ol{\Sigma}^{\Zar}=\bigcup_{i=1}^k b_i+B_i,
\end{equation} 	
and such that
\begin{equation}\label{lem2024eq2}
\text{each~$b_i+B_i$ is an irreducible component of~$\ol{\Sigma}^{\Zar}$.}
\end{equation}
Then each $B_i$ is equal to $ \pi_1(B_i)\times\pi_2(B_i)$.
\end{lemma}
\begin{proof} We first reduce the proof to the case~$k=1$ and~$b_i=0$. For each~$i\in\{1;\ldots;k\}$, define~$\Sigma_i:=\Sigma\cap (b_i+B_i)$. Then~\eqref{lem2024eq2} implies that~$\ol{\Sigma_i}^{\Zar}=b_i+B_i$. In particular there exists an element~$\sigma_i\in \Sigma_i$. We then have~$\sigma_i+B_i=b_i +B_i$.
Then the set~$\Sigma'_i=\{\sigma-\sigma'_i\,|\,\sigma\in \Sigma_i\}$ satisfies~\eqref{lem2024eq1} and~$\ol{\Sigma'_i}^{\Zar}=B_i$.

It is enough to prove, for each~$i=1,\ldots,k$, the Lemma~\ref{lem:2024} for~$\Sigma=\Sigma'_i$.

We now assume~$k=1$ and~$b_i=0$ and write~$B:=B_1$.

For every~$(b_1,b_2)\in \Sigma$, we have~$\gcd\{\operatorname{ord}(b_1);\operatorname{ord}(b_2)\}=1$. There exists thus~$m_1,m_2\in\Z$ such that~$m_1\cdot\operatorname{ord}(b_1)+m_2\cdot\operatorname{ord}(b_2)=1$. Let~$m:=m_2\cdot \operatorname{ord}(b_2)$. Then~$m\equiv1\pmod{\operatorname{ord}(b_1)}$ and
\[
(\pi_1(b_1,b_2),0)=(b_1,0)=m\cdot (b_1,b_2)\in m\cdot B=B.
\]
We deduce~$\pi_1(\Sigma)\times \{0\}\subseteq B$.
We have
\[
\ol{\pi_1(\Sigma)}^{\Zar}
=
\ol{\pi_1(\ol{\Sigma}^{\Zar})}^{\Zar}
=
\ol{\pi_1(B)}^{\Zar}=\pi_1(B).
\]
Thus~$\pi_1(B)\times\{0\}\subseteq B$. Similarly~$\{0\}\times\pi_2(B)\subseteq B$. In particular
\[
\pi_1(B)\times\pi_2(B)
=
\pi_1(B)\times\{0\}
+
\{0\}\times\pi_2(B)
\subseteq B+B=B.
\]
Obviously~$B\subseteq \pi_1(B)\times\pi_2(B)$.
We conclude~$B=\pi_1(B)\times\pi_2(B)$.
\end{proof}
\subsubsection{General case of Theorem~\ref{pureppart}}\label{finalproof}
Recall that we have isomorphisms
\begin{align}
A_{\operatorname{tors}}&\simeq
\left(
\bigoplus_{\ell\neq p}A[\ell^{\infty}]\right)\oplus A[p^\infty]
\\
&\simeq
\left(
\sum_{\ell\neq p}A[\ell^{\infty}]\right)\oplus \left(\Lambda+\Lambda'\right).
\end{align}
Accordingly, we can decompose each element of the sequence of torsion points~$a_n=a'_n+\lambda_n+\lambda'_n$. For every~$\sigma\in \Gal(\ol{K}/K)$ and every~$n$, we have
\[
\pi_{\mathfrak{p}}(\sigma(a_n-\lambda_n))=
\pi_{\mathfrak{p}}(\sigma(a_n))-
\pi_{\mathfrak{p}}(\sigma(\lambda_n))=
\pi_{\mathfrak{p}}(\sigma(a_n))
\]
and thus~$E_n=\pi_\mathfrak{p}(\Gal(\ol{K}/K)\cdot a_n)=\pi_\mathfrak{p}(\Gal(\ol{K}/K)\cdot( a_n-\lambda_n))$.
In proving Theorem~\ref{pureppart}, we may thus assume~$\lambda_n=0$ for every~$n$.

We consider~$\wt{a}_n:=(a'_n,\lambda'_n)\in A\times A$ and~$\wt{E}_n:=\pi_\mathfrak{p}(\Gal(\ol{K}/K)\cdot \wt{a}_n)$. Let~$\alpha:A\times A\to A$ be the addition map. We have
\[
\alpha(\wt{E}_n)=E_n.
\]
Since~$\alpha$ is a proper map, we have
\[
\alpha\left(\ol{\bigcup_{n\in\Z_{\geq0}} \wt{E}_n}^{\Zar}\right)=
\ol{\bigcup_{n\in\Z_{\geq0}} \alpha(\wt{E}_n)}^{\Zar}=
\ol{\bigcup_{n\in\Z_{\geq0}} {E}_n}^{\Zar}.
\]
Therefore, it is enough to prove Theorem~\ref{pureppart} for the sequence~$\wt{a}_n$ instead of the sequence~$a_n$.

Thanks to Theorem~\ref{mm}, we have
\[
\ol{\{\wt{a}_n\,|\,n\in\Z_{\geq0}\}}^{\Zar}=\bigcup_{i=1}^k d_i+D_i\subset A_K\times A_K
\]
for some~$d_1,\ldots,d_k\in A(\ol{K})_{\operatorname{tors}}\times A(\ol{K})_{\operatorname{tors}}$ and abelian subvarieties~$D_1,\ldots, D_k\subset (A\times A)_{\ol{K}}$. We assume that each~$d_i+D_i$ is actually a component of~$\ol{\{\wt{a}_n\,|\,n\in\Z_{\geq0}\}}^{\Zar}$ (the Zariski closure here is taken in $A_K\times A_K$). For every~$i$, we may assume that~$d_i\in \bigcup_{n\in\Z_{\geq0}} \Gal(\ol{K}/K)\cdot \wt{a}_n\subseteq\left(\bigoplus_{\ell\neq p}A[\ell^{\infty}]\right)\times \Lambda'$ and
\begin{equation}\label{pasdenom}
\forall n\in\Z_{\geq0},i\in\{1;\ldots;k\},\wt{a}_n-d_i\in \left(\bigoplus_{\ell\neq p}A[\ell^{\infty}]\right)\times \Lambda',
\end{equation}
 where $\Lambda '$ was introduced in \eqref{def:lambda prime}.

By Lemma~\ref{lem:2024}, we have, for each $i=1, \dots, k$,  $D_i=\pi_1(D_i)\times \pi_2(D_i)$.

We define
\begin{displaymath}
\Delta_i:=\pi_1(D_i)_{\mathfrak{p}}, \ \Delta'_i:=\ol{\pi_2(D_i)_{\mathfrak{p}}[p^{\infty}]}^{\Zar}, \delta_i:=\pi_\mathfrak{p}(d_i).
\end{displaymath}

Theorem~\ref{pureppart} for the sequence~$(\wt{a}_n=(a'_n,\lambda'_n))_n$ is a direct consequence of the following. 
\begin{prop}\label{prop:2024Claim}
We have
\begin{equation}\label{2024Claim}
\ol{\bigcup_{n\in\Z_{\geq0}} \wt{E}_n}^{\Zar}=\bigcup_{i=1}^{k} \delta_i+\Delta_i\times\Delta'_i.
\end{equation}
\end{prop}
\begin{proof}[Proof of Proposition~\ref{prop:2024Claim}]
We can find a finite extension $K'/K$, such that
\begin{equation}
d_1,\ldots,d_k\in A(K')
\end{equation}
and
\begin{equation}\label{2024ell-indep}
\forall\sigma',\sigma''\in \Gal(\ol{K}/K'),\exists\sigma\in \Gal(\ol{K}/K'), \forall n\in\Z_{\geq0}, \sigma(a'_n,\lambda'_n)=(\sigma'(a'_n),\sigma''(\lambda'_n)).
\end{equation}
The latter is an ``$\ell$-independence property'' and follows from \cite[N.136, \S2, Th. 1 p. 34]{SeOe4}. Since, as usual, to prove the claim we may replace $K$ by a finite extension, we simply denote $K'$ by $K$.

The inclusion~$\ol{\bigcup_{n\in\Z_{\geq0}} \wt{E}_n}^{\Zar}\subseteq\bigcup_{i=1}^{k} \delta_i+\Delta_i\times\Delta'_i$
is immediate. We need to prove that, for every~$i\in\{1;\ldots;k\}$, we have
\begin{equation}\label{toproveinclaim2024}
\delta_i+\Delta_i\times\Delta'_i\subseteq \ol{\bigcup_{n\in\Z_{\geq 0}} \wt{E}_n}^{\Zar}.
\end{equation}
We fix~$i$ in~$\{1;\ldots;k\}$.
In proving~\eqref{toproveinclaim2024}, we can always pass to a subsequence of~$\wt{a}_n$.
Without loss of generality, we may assume that~$\dim(D_i)>0$. Then passing to an infinite subsequence of~$\wt{a}_n$,
we may assume that~$\wt{a}_n\in d_i+D_i$ for every~$n$, and that~$\wt{a}_n$ is a generic sequence in~$d_i+D_i$.
By~\eqref{pasdenom}, we may translate by~$-d_i$ and assume~$d_i=0$.
Therefore the sequences~$\pi_1(\wt{a}_n)$ and~$\pi_2(\wt{a}_n)$ are generic in~$\pi_1(D_i)$ and~$\pi_2(D_i)$ respectively. 

By Corollary~\ref{cor 1.2.5}, the subset
\[
\bigcup_{n\in\Z_{\geq0}}\pi_{\mathfrak{p}}(\pi_1(\Gal(\ol{K}/K)\cdot \wt{a}_n))\subseteq \Delta_i
\]
is Zariski dense. Passing to a subsequence, there exists a sequence~$\sigma'_n$ in~$\Gal(\ol{K}/K)$
such that $\pi_{\mathfrak{p}}(\pi_1(\sigma'_n(\wt{a}_n)))$ is generic in~$\Delta_i$.

Proposition~\ref{prop:particular} shows that
\[
\bigcup_{n\in\Z_{\geq0}}\pi_{\mathfrak{p}}(\pi_2(\Gal(\ol{K}/K)\cdot \wt{a}_n))\subseteq \Delta'_i
\]
is Zariski dense in $\Delta'_i$. Passing to a subsequence, there exists a sequence~$(\sigma''_n)n$ in~$\Gal(\ol{K}/K)$
such that $(\pi_{\mathfrak{p}}(\pi_2(\sigma''_n(\wt{a}_n))))_n$ is generic in~$\Delta'_i$.

By~\eqref{2024ell-indep}, there exist~$\sigma_n\in \Gal(\ol{K}/K)$ be such that
\[
\forall n, \sigma_n(a'_n,\lambda'_n)=(\sigma'_n(a'_n),\sigma''_n(\lambda'_n)).
\]

By Theorem~\ref{mainthm}, we can write
\[
\ol{\bigcup_{n\in\Z_{\geq0}} \wt{E}_n}^{\Zar}=\bigcup_{j=1}^m \beta_j+\mathfrak{B}_j,
\]
and by Lemma~\ref{lem:2024}, we have
\[
\forall j\in\{1;\ldots;m\}, \mathfrak{B}_j=\pi_1(\mathfrak{B}_j)\times \pi_2(\mathfrak{B}_j).
\]
Passing to an infinite subsequence, there exists~$j$ such that
\[
\forall n,\pi_{\mathfrak{p}}(\sigma_n(\wt{a}_n))\in \beta_j+\mathfrak{B}_j.
\]
Thus, the sequence~$\pi_{\mathfrak{p}}(\pi_1(\sigma_n(\wt{a}_n)))=\pi_{\mathfrak{p}}(\sigma'_n(a'_n))$ in~$\pi_1(\beta_j+\mathfrak{B}_j)$
is dense is~$\Delta_i$. Thus
\[
\Delta_i\subseteq \beta_j+\mathfrak{B}_j.
\]
On the other hand, we have~
\[
\beta_j+\mathfrak{B}_j\subseteq \bigcup_{j=1}^m \beta_j+\mathfrak{B}_j=\ol{\bigcup_{n\in\Z_{\geq0}} \wt{E}_n}^{\Zar}\subseteq \Delta_i\times\Delta'_i.
\]
Thus~$\pi_1(\mathfrak{B}_j)=\Delta_i=\pi_1(\beta_j+\mathfrak{B}_j)$. Similarly,~$\pi_2(\mathfrak{B}_j)=\Delta'_i=\pi_2(\beta_j+\mathfrak{B}_j)$. Therefore
\[
\Delta_i\times\Delta'_i=\pi_1(\mathfrak{B}_j)\times \pi_2(\mathfrak{B}_j)=\mathfrak{B}_j
\subseteq \ol{\bigcup_{n\in\Z_{\geq0}} \wt{E}_n}^{\Zar}.
\]
This proves~\eqref{toproveinclaim2024} and concludes the proof of Proposition~\ref{prop:2024Claim}.
\end{proof}
This implies Theorem~\ref{pureppart} for the sequence~$\wt{a}_n$, and therefore Theorem~\ref{pureppart} for the sequence~$a_n$.

\appendix
\section{Proof of Proposition \ref{thingy statement}}\label{preliminaries}
\subsection{Some \texorpdfstring{$p$}{p}-adic Hodge theory (for abelian varieties with good reduction)}
To conclude the paper we have to prove Proposition \ref{thingy statement}, which was used in Theorem \ref{thm502}. Here we try to make our argument as understandable as possible, by recalling some standard facts, but the reader fluent in $p$-adic Hodge theory may simply skip to the proof given in Section \ref{finalsectionp}.
\subsubsection{Notation}\label{thingy notations}
Let $p$ be a rational prime number, $L$ be a finite extension of $\Q_p$, $\Oo_L$ its ring of integers and $k$ its residue field. Let $\mathcal{A}$ be an abelian scheme over $\Oo_L$, of relative dimension $g>0$, and denote its special and generic fibres respectively by $A_k$ and $A_L$. We choose $\overline{L}$ an algebraically closed extension of $L$ and denote $\overline{k}$ 
the corresponding residue field extension of $k$, so that we can define the $\Z_p$-linear Tate modules:
\begin{equation}\label{thingy T module}
T_p(A_k):=\Hom(\Q_p/\Z_p,A_k[p^\infty](\kbar))\text{ and }T_p(A_L):=\Hom(\Q_p/\Z_p,A_L[p^\infty](\overline{L}))
\end{equation}
The former is the \emph{naïve} Tate module (accounting for the $p$-\emph{rank} of $A$), whereas $T_p(A_L)$ has rank $2g$.

We denote by $\QQbar$ the algebraic closure of $\Q$ inside of $\Lbar$, so that the field $\QQbar\cap \Q_p$ is properly defined.

We write the corresponding $\Q_p$-linear Tate modules as follows
\[
V_p(A_k)=T_p(A_k)\tens_{\Z_p}\Q_p\simeq \Hom(\Q_p,A_k[p^\infty](\kbar)),
\]
and 
\[
V_p(A_L)=T_p(A_L)\tens_{\Z_p}\Q_p\simeq \Hom(\Q_p,A_L[p^\infty](\Lbar)).
\]

We are interested in the reduction map $A_L(\Lbar)\to A_k(\kbar)$ induced by
\[
A_L(\Lbar)\xleftarrow{\sim}\mathcal{A}(\Oo_{\Lbar})\xrightarrow{}A_k(\kbar).
\]

Passing to the Tate modules there is a reduction map $T_p(A_L)\to T_p(A_k),$
which is equivariant with respect to the covariant action of $\End(A_L)\subseteq \End(A_{\Lbar})$ whereas
\begin{displaymath}
\End(A_k)\subseteq \End(A_{\kbar})
\end{displaymath}
acts a priori only on $T_p(A_k)$.

In $\End(A_k)$ there is a specific element, the Frobenius of $A_k$ relative to $k$,
\begin{equation}\label{thingy frob of Ak}
\FrobAk:A_k\to A_k
\end{equation}
associated to the $q$-power map, where $q=|k|$. On coordinates $A(\kbar)\to A(\kbar)$ it acts with the Frobenius map
\[\Frob_{\kbar/k}:\kbar\xrightarrow{x\mapsto x^q} \kbar.\]
 (This is just the morphism of local ringed spaces given by the identity at the topological level and the $|k|$-power at the level of structural sheaves.)

\subsubsection{Witt vectors notations}
Two standard references for crystalline cohomology are  \cite{zbMATH03467279, zbMATH03595321}. Let $\Wittk$ be the ring of Witt vectors with coefficients in $k$ (and $L_0=\Q_q$ its fraction field), and $F:=\WittF,V:=\WittV\in \End_{W(\FF_p)}(W(k))$ the \emph{Frobenius} and
\emph{Verschiebung} operators. 

In practice: the first acts on the Witt coordinates by applying on each coordinate the $p$-th power map $\Frob_{k/\FF_p}$, 
the second acts by shifting those coordinates. 

%Concretely the ring $\Witt(\FF_p)$ is $\Z_p$ defined in a schematic way, whereas the ring $\Witt(\FF_q)$ is naturally isomorphic to $\Z_q$, the ring of integers of an unramified extension $\Q_q$ of $\Q_p$ with residue field $\FF_q$. One should know that
%the fraction field $\Witt(k)\tens_{\Z_p}\Q_p$
%is naturally isomorphic to the maximal unramified extension $L_0$ of $\Q_p$ contained in $L$.

Of relevance is the Dieudonné ring $\Dieukfull$ as the (non commutative) $\Wittk$-algebra of additive 
endomorphisms of $\Wittk$ generated by $\Wittk$ acting by ring multiplication, and $\WittF$ and $\WittV$ just defined. One usually recovers $V$ from $F\circ V=p=V\circ F$.

\subsection{Crystalline-étale comparison theorem}
On the category of $\Oo_L$-abelian schemes $\mathcal{A}$, we consider the crystalline and $p$-adic étale cohomologies 
as contravariant functors
\begin{equation}\label{Cohomologies introduction}
\CoHcrys(A_k)\text{ and }\CoHet(A_{\Lbar},\Z_p).
\end{equation}

The first is a free $\Wittk$-module of rank $2g$ functorial in the special fibre $A_k$, and comes equipped with
a $\Wittp$-linear endomorphism $\phi\in \End_{\Wittp}(\CoHcrys(A_k))$ which is semi-linear with
respect to $\WittF$: mapping $\WittF$ to $\phi$ makes $\CoHcrys(A_k)$ a $\Dieukfull$-module which is free as a $\Wittk$-module.
 
%The second is a free $\Z_p$-module of rank $2g$, and is functorial in the generic fibre $A_L$. %, and depends upon a choice of algebraically closed extension $\Lbar/L$. It is in guise the usual Tate module, or rather its dual.

The following comparison theorem holds for whichever of the Fontaine rings $E=B_{\text{dR}}$ (and then generalises to varieties, according to Faltings) or, sufficiently for abelian schemes, $E=B_{\crys}$. Those constructions depend a priori on $L$ and $\Lbar$, but the resulting rings not, up to the relevant isomorphism. The construction of period rings, with all details, can be found in \cite{zbMATH03865471}. We following is a specail case of the so called $C_{\operatorname{crys}}$-theorem (for $H^1s$ and abelian varieties with good reduction). %Of course everything holds in a greater generality, even if we limit our discuss to abelian varieties. 
The general case is due to Faltings \cite{zbMATH00002230}.

\begin{thm}\label{abelian comparison}
There is a field extension $E$ of the fraction field $L_0$ of $W(k)$ and a \emph{natural} isomorphism
\begin{equation}\label{eqcrys2}
\gamma_{\operatorname{crys}}(A): E\otimes_{W(k)} \CoHcrys(\mathcal{X}_k)\xrightarrow{\sim} E\otimes_{\Z_p} \CoHet(\mathcal{X}_{\Lbar} ,\Z_p).
\end{equation}
\emph{of functors} from the category of smooth projective varieties $\mathcal{X}$ over $\Oo_L$
to that of $E$-vector spaces.
\end{thm}
The next remark explains the reason why we need to consider $B_{\operatorname{crys}}$-coefficients (rather than $\C_p$).
\begin{rmk} 
Originally Tate \cite{zbMATH03252949} proved the so called Hodge-Tate decomposition for the first étale cohomological group of abelian varieties. That is that the existence of a canonical equivariant isomorphism
\begin{displaymath}
H^1_{\text{et}}(A_{\overline{L}},\Q_p)\otimes \C_p \cong H^1(A,\Oo_A)\otimes \C_p(-1)\oplus H^0(A,\Omega^1)\otimes \C_p.
\end{displaymath}
By the degeneration of the Hodge to de Rham spectral sequence in characteristic zero, the RHS is isomorphic to $H^1_{dR}(A/\C_p)$. The problem is that an such isomorphism is, in general, not canonical. 
\end{rmk}
On the homology side, we get the covariant $\Wittk$-dual and $\Z_p$-dual functors
\[
\Hcrys(A_k):=\Hom_{\Witt(k)}(\CoHcrys(A_k),\Wittk)
\text{ and }
\Het(A_{\overline{L}}):=\Hom_{\Z_p}(\CoHet(A_{\overline{L}},\Z_p))
\]
and, passing to the $E$-linear duals, a natural isomorphism
\begin{equation}\label{eqcrys3}
\gamma^\vee_{\operatorname{crys}}(A): E\otimes_{W(k)} H_1^{\operatorname{crys}}(A_k)\xleftarrow{\sim} E\otimes_{\Z_p} H_1^{\operatorname{et}}(A_{\overline{L}},\Z_p).
\end{equation}

\subsection{The connected-étale exact sequence}
We consider the reduction map
\begin{equation}\label{thingy reduction map}
\reduction:A(\overline{L})[p^\infty]\to A(\overline{k})[p^\infty],
\end{equation}
in terms of the connected-étale exact sequence. 

We first recall some notions from $p$-divisible groups and Dieudonné theory, following Demazure's classical reference \cite{zbMATH03390951} (both over $L$ and $k$). To the abelian scheme $\mathcal{A}$, and the abelian varieties $A_L$ and $A_k$, Tate attached \emph{$p$-divisible groups}
\[
\Gcal_{\mathcal{A}}=\varinjlim_{n\in{\Z_{\geq0}}} \mathcal{A}[p^n]\hookrightarrow \mathcal{A}[p^{n+1}],
\]
and its generic and special fibres
\[
\Gcal_{A_L}=\varinjlim_{n\in{\Z_{\geq0}}} A_L[p^n]\hookrightarrow A_L[p^{n+1}]
\text{ and }
\Gcal_{A_k}=\varinjlim_{n\in{\Z_{\geq0}}} A_k[p^n]\hookrightarrow A_k[p^{n+1}].
\]
These are \emph{ind-schemes} of group schemes: its functor of points is tested against $\Oo_L$-algebras $R$ and is $\Gcal_{\mathcal{A}}(R)=\mathcal{A}(R)[p^n]$.

We denote the neutral connected components (with respect to the union of Zariski topologies) as
$\Gcal_{\mathcal{A}}^0$, $\Gcal_{A_L}^0$, and $\Gcal^0_{A_k}$. The first and third are actually the formal Lie groups attached to $\mathcal{A}$ and $A_k$. %(The middle one $\Gcal^0_{A_L}$ is $0$, whereas $\widehat{A_L}$ has dimension $g$, and won't be of interest.)

The corresponding quotients are well defined as $p$-divisible groups and there are short exact sequences of $p$-divisible groups:
\begin{eqnarray}
\label{G Acal exact sequence}
&0\xrightarrow{}\Gcal_{\mathcal{A}}^0\xrightarrow{\iota}\Gcal_{\mathcal{A}}\xrightarrow{\varpi}\Gcal_{\mathcal{A}}^{\etale}\xrightarrow{}0
&\text{ over }\Oo_L,\\
\label{G Ak exact sequence}
\text{ and }
&0\xrightarrow{}\Gcal_{A_k}^0\xrightarrow{\iota_k}\Gcal_{A_k}\xrightarrow{\varpi_k}\Gcal_{A_k}^{\etale}\to 0
&\text{ over }k.
\end{eqnarray}

The (co)homologies we evoked in \eqref{Cohomologies introduction} factor as functor through the ``passing to the attached $p$-divisible groups over $\Oo_L$'',
and $\CoHcrys$, resp. $\CoHet$, factors even more through ``passing to the special (resp. generic) fibre''.
We can moreover identify respectively
\[
\Hcrys(A_k)\text{ and }
\Het(A_{\Lbar};\Z_p)
\]
with the \emph{covariant Dieudonné module} $M(\Gcal_{A_k})$ of $\Gcal_{A_k}$, and its \emph{Tate module}
\[
T(\Gcal_{A_{\Lbar}})=\Hom(\Q_p/\Z_p,\Gcal_{A_L}(\Lbar)).
\]
%As Dieudonné modules description is fairly explicit by reduction to the finite group schemes $A_k[p^n]$, this can serve as a working definition of the $\CoHcrys$ and $\CoHet$ cohomologies applied to abelian varieties.

It is important that the Dieudonné module (resp. Tate module) makes sense for any the $p$-divisible groups appearing in \eqref{G Ak exact sequence} and \eqref{G Acal exact sequence}
respectively. Even though these are not the full $p$-divisible group of an abelian variety, or of an abelian scheme resp. We still get covariant functors
\[
\mathbf{M}( )\text{ and }\mathbf{T}( )
\]
on the category of $p$-divisible groups over $k$ and $L$ respectively, to that of free $\Wittk$-modules and $\Z_p$-modules respectively.

Applying the functors $\mathbf{T}( )$ and $\mathbf{M}( )$ to \eqref{G Acal exact sequence} and \eqref{G Ak exact sequence} respectively gives us exact sequences
that we relabel
\begin{align}\label{notations T}
&0\xrightarrow{}T^0\xrightarrow{}T\xrightarrow{}T^{\etale}\xrightarrow{}0 
=&
&0\xrightarrow{}\mathbf{T}(\Gcal_{\mathcal{A}}^0)\xrightarrow{\mathbf{T}(\iota)}
	\mathbf{T}(\Gcal_{\mathcal{A}})
		\xrightarrow{\mathbf{T}(\varpi)}\mathbf{T}(\Gcal_{\mathcal{A}}^{\etale})\xrightarrow{}0
\\ \label{notations M}
\text{ and }
&0\xrightarrow{}M^0\xrightarrow{}M\xrightarrow{}M^{\etale}\xrightarrow{}0
=&
&0\xrightarrow{}\mathbf{M}(\Gcal_{A_k}^0)\xrightarrow{\mathbf{M}(\iota_k)}\mathbf{M}(\Gcal_{A_k})\xrightarrow{\mathbf{M}(\varpi_k)}\mathbf{M}(\Gcal_{A_k}^{\etale})\to 0.
\end{align}

\subsubsection*{Fact} We will use the following explicit description of the Dieudonné module of an étale $p$-divisible group $\Gcal$ over $k$.
We have, as $W(k)$-modules, a natural identification
\[
\mathbf{T}(\Gcal)\tens_{\Z_p} W(k)\simeq \mathbf{M}(\Gcal).
\]
We won't need to explicit semi-linear action making it a \emph{Dieudonné} module, but we stress that the above identification is functorial, and in
particular receives a left action of $\End(\Gcal)$.

The reduction map $\reduction$ factors through the above exact sequences, by which we mean the commutativity of the square
\[
\begin{tikzcd}
\Gcal_{\mathcal{A}}[p^\infty](\Oo_{\overline{L}})
 \arrow{d}{\reduction} 
 \arrow{r}{\varpi}
&
\Gcal_{\mathcal{A}}^{\etale}[p^\infty](\Oo_{\overline{L}}) \arrow{d}{\sim}
\\
\Gcal_{A_k}[p^\infty](\overline{k}) \arrow[equal]{r}{\varpi_k}
&
\Gcal_{A_k}^{\etale}[p^\infty](\overline{k}).
\end{tikzcd}
\]

Applying the dualising $\Hom(\Q_p/\Z_p,-)$ to pass to the Tate modules, we get equivalently
\[
\begin{tikzcd}
T \arrow[d] \arrow[r]
	 & T^{\etale} \arrow[d, "\sim"]\\
T_p({A_k}[p^\infty](\overline{k})) \arrow[equal]{r}
	& T_p({A_k}[p^\infty](\overline{k})).
\end{tikzcd}
\]
where in the bottom row is the naïve Tate module of $A_k$: its rank as a free $\Z_p$ module is the $p$-\emph{rank} of $A_k$.

In another formulation, we have the identity
\begin{equation}\label{thingy ker identity}
\ker\bigl(\mathbf{T}(\varpi):T\to T^{\etale}\bigr)=\ker\bigl(T(\reduction):T\to T_p({A_k}[p^\infty](\overline{k})) \bigr)
\end{equation}
\subsection{Unit root eigenspaces} We study the connected-étale exact sequence of Dieudonné modules:
\[
0\xrightarrow{}M^0\xrightarrow{}M\xrightarrow{}M^{\etale}\xrightarrow{}0.
\]

\subsubsection{Dieudonné-Manin semilinear eigenspaces}We recall a notion of eigenspaces for semilinear actions. %See also Example 8.1.3 in Brinon, Conrad: p-adic hodge theory. This section could be made shorter??
Let $M'$ be any free Dieudonné module over $k$, namely a finite rank free module over $\Wittk$ equipped with a $\Wittp$-linear endomorphism $\phi$ which is semi-linear with respect to $\Wittk$. We have the following notion of eigenspace on the $L_0$-linear vector space $D=M'\tens_{\Wittk}L_0$. Pick any $\Wittk$ basis of $M'$
over $L_0$, the representative matrix $\Phi$ of $\phi$ applied to this basis has associated (generalised) eigenspaces in $D\tens_{L_0}\overline{L}$, and these
are defined over $L_0$ when the corresponding eigenvalue belongs to $L_0$. A change of basis of matrix $B$ will change $\Phi'=\Phi\mapsto \WittF(B)\Phi B^{-1}$, each eigenvalue obtained from the $\lambda$ by some operation
\[
\lambda\mapsto\lambda'=\WittF(c_\lambda)/c_\lambda\cdot \lambda
\]
for some invertible $c(\lambda)$ in $\Wittk \subseteq {L_0}^{\text{nr}}$. As $\WittF$ preserves the $p$-adic absolute values $\abs{ }_p:L_0\mapsto \R_{\geq 0}$, we will have the invariance
\[
\abs{\lambda}_p=\abs{\lambda'}_p.
\]
For every real number $\mu$, we let $\overline{V}_\mu$ be the sum, in $\overline{D}=D\tens_{L_0}\overline{L}$, of the generalised eigenspaces for eigenvalues $\lambda$ having valuation $\mu=-\log_p\abs{\lambda}_p$. Then $\overline{D}_\mu$ does not depend on the choice of basis, and is defined over $L_0$.

This can be approached another way. Denote $f=[k:\FF_p]$, so that $\Wittk^f$ will actually be the identity of $\Wittk$.  But the power $\phi^f$, since $\Wittk^f$-semi-linear, will actually be $\Wittk$-linear. We let $P(X)$ be its characteristic polynomial, and $P_\mu(X)$ be, in $\overline{L}[X]$, its maximal unitary factor whose roots with roots $\lambda$ have valuation $f\mu=-f\cdot\log\abs{\lambda}_p$. 
Then $P_\mu(X)$ is actually defined over $L_0$,  and the kernel $D_\mu=\ker P_\mu(\phi)$ will produce $\overline{D}_\mu=D_\mu\tens_{L_0}\Lbar$.

These ``Dieudonné modules'' $M'$ are also known as \emph{crystals}, the corresponding $D$ as \emph{iso}crystals, and the subspace $D_0$ as a \emph{unit root} isocrystal.

\subsubsection{Unit root isosplitting of étale quotient} It happens that the exact sequence 
\[
\begin{tikzcd}
    	0\arrow{r} 
	&M^0\tens_{W(k)}L_0\arrow{r}{t}
	& M\tens_{W(k)}L_0\arrow{r}{\pi} 
	& M^{\etale}\tens_{W(k)}L_0\arrow{r} \arrow[l, "s"', bend right=33]
	& 0
\end{tikzcd}
\]
is (uniquely) split as an exact sequence of ($\WittF$-semi-linear) modules, and that a (unique semi-linear) splitting is given by the unit root subspace
\[
s(D^{\etale})=D_0\text{ of }D=M\tens_{W(k)}L_0\text{, letting }D^{\etale}:=M^{\etale}\tens_{W(k)}L_0.
\]
\subsubsection{Spectral projector $\pi$ and functorial Frobenius} Let $\pi$ be the corresponding eigenspace projector $
D\to D$ 
with kernel $D_0$. It can be recovered in terms of $\phi^f$ in the following usual way. Let $K_0\subseteq L_0$ be any field containing
the coefficients of the factor $P_\mu(X)$ of the characteristic polynomial $P(X)$ of $\phi^f$. 
There is a polynomial $Q(X)\in K_0[X]$, uniquely defined up to $P(X)$, such that $Q\equiv 0\pmod{P_\mu(X)}^?$ and $Q\equiv 1\pmod{P_\phi(X)/P(X)}^?$. %This is insured by a variant of Chinese remainders for the coprime $P_\mu(X)$ and $P(X)/P_\mu(X)$; a direct formula in term of the spectrum is (ref); an algorithm is (ref). 
For any such polynomial, $\pi=Q(\phi^f)$ is the spectral projector with kernel $D_0$.

We now consider the $q$-th powers Frobenius endomorphism $\FrobAk\in\End(A_k)$ from \eqref{thingy frob of Ak}
of the abelian variety $A_k$ with respect to the field $k$. We let $\End(A_k)$ act on $\Gcal_{A_k}$ as well.

Applying the functor $\mathbf{M}( )$ to it turns out to give again
\begin{equation}\label{Identity of FrobAk phi}
\mathbf{M}(F_{A_k})=\phi^f\text{ on }M.
\end{equation}
We now define $\pi:=Q(\Frob_{A_k})\in \End(A_k)\tens_\Q K$.

We may extend $D:\End(A_k)\to \End_{L_0}(D(\Gcal_{A_k}))$ by $L_0$-linearity, and now define
\[
\mathbf{D}_{K_0}: \End(A_k)\tens_\Q K_0\to \End(A_k)\tens_\Q L_0\to \End_{L_0}(D(\Gcal_{A_k})) 
\]
and finally the element
\[
\mathbf{D}_{K_0}(\pi) \in \End_{L_0}(D) 
\]
is well defined, and \eqref{Identity of FrobAk phi} implies that $
\mathbf{D}_{K_0}(\pi)=Q(\phi^f)=\pi\text{ in }\End_{L_0}(D)$.

\subsubsection{Comparison Theorem for $p$-divisible groups} We now review the connected-étale sequence in our comparison by factoring, for abelian varieties, the crystalline étale comparison through the categories of $p$-divisible
groups on $\mathcal{O}_L$ and on $k$.
 
We refer to the dedicated book \cite{Fontainepdivbook} of Fontaine, and most notably Appendix B therein. A more recent exposition can be found for example in \cite[Chap. 4]{LaCourbe}. See also \cite[Sec. 6]{zbMATH01229029}. In \emph{op. cit.} Faltings describes more delicate problems about an \emph{integral comparison} and works with $B^+(\Oo_L)$, rather than $B_{\crys}$, but it is easy to see that those results imply the following. 
\begin{thm}[Comparison for $p$-divisible groups]\label{pdiv comparison}
Let $E$ be a period ring as in Theorem \ref{eqcrys2}. There exists a natural ($E$-linear) isomorphism
\begin{equation}\label{eqcrys div}
\gamma(\Gcal): E\otimes_{W(k)}\CoHcrys(\Gcal_k)\xrightarrow{\sim} E\otimes_{\Z_p} \CoHet(\Gcal_{\overline{L}},\Z_p).
\end{equation}
of functors from the category of $p$-divisible groups $\Gcal$ over $\Oo_L$
to that of vector spaces over $E$.
\end{thm}
\begin{rmk}
We actually know that $\gamma(\Gcal)$ extends the map $\gamma_{\operatorname{crys}}(A)$ from Theorem \ref{abelian comparison}.
\end{rmk}

We again use the $E$-linear dual comparisons
\[
\gamma^\vee(\Gcal):\mathbf{M}(\Gcal)\tens_{W(k)}E\xleftarrow{\sim}\mathbf{T}(\Gcal)\tens_{\Z_p} E
\]
or equivalently, with the isocristal ${D}(\Gcal)=\mathbf{M}(\Gcal)\tens_{\Wittk}L_0$ and the $\Q_p$-linear Tate module $V(\Gcal)=\mathbf{T}(\Gcal)\tens_{\Z_p}\Q_p$, with abuse of notations,
\[
\gamma^\vee(\Gcal):{D}(\Gcal)\tens_{L_0}E\xleftarrow{\sim}V(\Gcal)\tens_{\Q_p} E.
\]

\subsection{Assembling everything together}\label{finalsectionp} The following enhances Proposition \ref{thingy statement} into a more precise statement, which we are ready to prove.
\begin{prop}\label{thingy statement precise} Set $D_\pi:=\im \mathbf{D}_{K_0}(\pi)$ for the $L_0$-linear subspace of $D=\mathbf{D}(\Gcal_{A_k})$ given by the image of $\pi$.
We consider the comparison isomorphism
\[
\gamma_{\mathcal{A}}^\vee:=\gamma^\vee(\mathcal{A}):V\tens_{\Q_p} E \to D\tens_{L_0} E.
\]

\begin{enumerate}
\item Then $\gamma$ is equivariant with respect to the covariant (i.e. on the left) action of $\End(A_L)\tens_{\Z} E$ on both side given by
\begin{itemize}
\item the $\mathbf{M}( )$-functorial action of $\End(A_L)$ on the $\Z_p$-module $M$, extended $E$-linearly to $V\tens_{\Q_p} E$;
\item the composition of the reduction $\End(A_L)\xleftarrow{\sim}\End(A_L)\hookrightarrow \End(A_k)$ with
the $\mathbf{D}( )$-functorial action of $\End(A_k)$ on the $\Q_p$-linear vector space $D$, extended $E$-linearly.
\end{itemize}
\item With the notation of Proposition \ref{thingy statement}, we have the identity
\begin{equation}\label{thingy fundamental identity}
\gamma_{\mathcal{A}}^\vee (V^r\tens E)=D_c\tens_{L_0} E
\end{equation}
of $E$-linear vector subspaces of $D\tens_{L_0}E$.
\end{enumerate}
\end{prop}

\begin{proof} The first part is only an explicit formulation of the functoriality in $\mathcal{A}$ of the comparison isomorphisms,
that is that $\gamma^\vee$ is a morphism of \emph{functors}. This requires only the abelian varieties comparison recalled as Theorem \ref{abelian comparison}.

For $e\in\End(\mathcal{A})$ we apply $\gamma^\vee$ to $e$ to get the commutativity of

\[
\begin{tikzcd}
\phantom{D}\arrow{d}{\gamma^\vee(e)}
     & V\tens_{\Q_p} E\arrow{r}{\gamma^\vee(\Acal)}\arrow{d}{V(e)}& D\tens_{L_0} E\arrow{d}{D(e)}\\
\phantom{D}
	& V\tens_{\Q_p} E\arrow{r}{\gamma^\vee(\Acal)}& D\tens_{L_0} E.
\end{tikzcd}
\]

The second part involves the connected-étale sequence, $p$-divisible groups and the rest of the theory.
We apply $\gamma^\vee$
to the exact sequence 
\[
\begin{tikzcd}
    0\arrow{r} 
	& \Gcal^0_{\mathcal{A}} \arrow{r}{\iota}
	& \Gcal_{\mathcal{A}} \arrow{r}{\varpi}
	& \Gcal^{\etale}_{\mathcal{A}}\arrow{r}
	& 0
\end{tikzcd},
\]
 and get the isomorphism of exact sequences, in notations \eqref{notations T} and \eqref{notations M},
\[
\begin{tikzcd}
    0\arrow{r} 
	& T^0\tens_{\Z_p} E \arrow{r}{\mathbf{T}(\iota)\tens E}\arrow{d}{\gamma^\vee(\Gcal^0_{\mathcal{A}})}
	& T\tens_{\Z_p} E\arrow{r}{\mathbf{T}(\varpi)\tens E}\arrow{d}{\gamma^\vee(\Gcal_{\mathcal{A}})} 
	& T^{\etale}\tens_{\Z_p} E\arrow{r}\arrow{d}{\gamma^\vee(\mathcal{\Gcal}_{\mathcal{A}}^{\etale})} 
	& 0 \\
    0\arrow{r} 
	& M^0\tens_{W(k)} E\arrow{r}{\mathbf{M}(\iota)\tens E}
	& M\tens_{W(k)} E\arrow{r}{\mathbf{M}(\varpi)\tens E} 
	& M^{\etale}\tens_{W(k)} E\arrow{r}
	& 0.
\end{tikzcd}
\]
In particular the kernel of $V(\iota)\tens E$ is sent to the kernel of $D(\iota)\tens E$. We have to check that
the equality $\gamma(\ker V(\varpi)\tens E)=\ker(D(\varpi)\tens E)$ \emph{is} the equality \eqref{thingy fundamental identity} we seek.

We can already determinate the upper kernel before passing to $E$ from \eqref{thingy ker identity} via identity of sequences:
\[
\begin{tikzcd}
    0\arrow{r} 
	& T(\Gcal^0_{\mathcal{A}}) \arrow{r}{T(\iota)}\arrow[equal]{d}
	& T(\Gcal_{\mathcal{A}}) \arrow{r}{T(\text{red})}\arrow[equal]{d} 
	& T(\Gcal^{\etale}_{\mathcal{A}})\arrow{r}\arrow{d}{\sim} 
	& 0 \\
    0\arrow{r} 
	& T^0\arrow{r}
	& T(A_L)\arrow{r} {T(\varpi)}
	& T^{\etale}\arrow{r}
	& 0.
\end{tikzcd}
\]
Thanks to \eqref{thingy ker identity}, we can identify $T^0$ with $\ker(T(\reduction))$.

Likewise, we determinate, the lower kernel before passing from $\Wittk$ to $E$, from the sequence
\[
\begin{tikzcd}
    0\arrow{r} 
	& M^0\arrow{r}{\mathbf{\iota}}
	& M \arrow{r}{\mathbf{\varpi}}
	& M^{\etale}\arrow{r}
	& 0.
\end{tikzcd}
\]
We note that $\End(\Gcal_{k})$ acts $\Z_p$-linearly on this sequence: it acts on each term and the maps are functorial.
The kernel of idempotent (to prove) $\mathbf{D}_{K_0}(\pi)$ is the unit root lift $D_0\subseteq D$ of $D^{\etale}$.
We want to check that its image $D_\pi=\mathbf{D}_{K_0}(\pi)(D)$ is indeed the image of $D^0:=M^0\tens_{\Wittk}L_0$. 
By exactness and dimensionality, it amounts to show that the image of $D_\pi$ in $D^{\etale}$ is $\{0\}$.

The map $\Gcal_{A_k}\to\Gcal^{\etale}_{A_k}$ is $\End(A_k)$-equivariant: namely, the connected component of $\Gcal_{A_k}$ is invariant under $\End(A_k)$.
By $\mathbf{M}( )$ funtoriality so is the map $M\to M^{\etale}$, and equivariance extends $L_0$-linearly to an action of $\End(A_k)\tens_\Q L_0$ on $D\to\ D^{\etale}$.

We asserted that the unit root $D_0\subseteq D$ is a linear lift of $D\to D^{\etale}$. By construction, $D_0$ is invariant under $\mathbf{M}(\FrobAk)$.
It follows the restriction map $D_0\to D^{\etale}$ is $L_0[\FrobAk]$-equivariant. Since $\mathbf{D}_{K_0}(\pi)=Q(\mathbf{M}(\FrobAk))$ will
be $0$ on $D_0$ hence, by equivariance, on $D^{\etale}$. This ends the proof of Proposition \ref{thingy statement precise}.
\end{proof}

%\begin{rmk}
%We provide an alternative and direct explanation of the fact that $Q(\FrobAk)$ acts on $D^{\etale}$ by $0$. It relies on the remark that $\FrobAk$ is invertible when acting on $\Gcal^{\etale}$. Hence the action $F'$ of $\FrobAk$ is a $\ZZ_p$-linear automorphism of $M^{\etale}$:
%its eigenvalues will be invertible algebraic integers over $\ZZ_p$, hence $p$-adic units. Furthermore, the action of $F'$ on $D^{\etale}$
%being a quotient of that on $D$, the former's characteristic polynomial, say $P'$ will be a factor of $P$, that of the later. 
%As roots are $p$-adic units, it will be furthermore a factor of $Q$, by definition of $Q$. It follows $Q(F')$, which gives the action of $\pi$ on $D^{\etale}$, is indeed $0$.
%\end{rmk}

\bibliographystyle{abbrv}
\bibliography{biblio.bib}

\Addresses

\end{document}